\documentclass[leqno,10pt,a4paper]{amsart}
\usepackage[margin=2.5cm]{geometry}
\usepackage[usenames,dvipsnames]{color}
\usepackage{hyperref}
\hypersetup{
 colorlinks,
 linkcolor={red!50!black},
 citecolor={blue!50!black},
 urlcolor={blue!80!black}
}
\usepackage{tikz}
\usepackage{graphicx}
\usepackage{mleftright}
\usepackage{booktabs}
\usepackage{enumitem}

\setlist[enumerate]{labelsep=*, leftmargin=1.5pc}
\setlist[enumerate]{label=\normalfont(\roman*), ref=\roman*}
\newtheorem{thm}{Theorem}[section]
\newtheorem{lemma}[thm]{Lemma}
\newtheorem{cor}[thm]{Corollary}
\newtheorem{prop}[thm]{Proposition}
\theoremstyle{definition}
\newtheorem{example}[thm]{Example}

\newtheorem{definition}[thm]{Definition}
\newtheorem{conjecture}[thm]{Conjecture}
\numberwithin{equation}{section}
\newcommand{\Z}{\mathbb{Z}}
\newcommand{\Q}{\mathbb{Q}}
\newcommand{\R}{\mathbb{R}}
\newcommand{\C}{\mathbb{C}}

\newcommand{\Proj}{\mathbb{P}}
\newcommand{\abs}[1]{\left\vert{#1}\right\vert}
\newcommand{\dual}[1]{{#1}^*}
\DeclareMathOperator{\Disc}{\Delta}
\DeclareMathOperator{\Ehr}{Ehr}
\newcommand{\vol}[1]{\operatorname{vol}\mleft({#1}\mright)}
\newcommand{\bdry}[1]{\partial{#1}}
\newcommand{\intr}[1]{{#1}^\circ}
\newcommand{\V}[1]{\operatorname{vert}\mleft({#1}\mright)}
\renewcommand{\dim}[1]{\operatorname{dim}\mleft({#1}\mright)}
\renewcommand{\det}[1]{\operatorname{det}\mleft({#1}\mright)}
\renewcommand{\gcd}[1]{\operatorname{gcd}\mleft\{{#1}\mright\}}
\renewcommand{\Re}[1]{\operatorname{\mathfrak{Re}}\mleft({#1}\mright)}
\renewcommand{\Im}[1]{\operatorname{\mathfrak{Im}}\mleft({#1}\mright)}
\newcommand{\sconv}[1]{\operatorname{conv}\mleft\{{#1}\mright\}}
\newcommand{\scone}[1]{\operatorname{cone}\mleft\{{#1}\mright\}}
\newcommand{\mult}[1]{\operatorname{mult}\mleft({#1}\mright)}
\DeclareMathOperator{\Hom}{Hom}
\begin{document}
\author[G.\,Heged{\"u}s]{G{\'a}bor~Heged{\"u}s}
\address{\'{O}buda University\\Antal Bejczy Center for Intelligent Robotics\\Budapest\ H-1032\\Hungary}
\email{hegedus.gabor@nik.uni-obuda.hu} 
\author[A.\,Higashitani]{Akihiro Higashitani}
\address{Department of Mathematics\\Graduate School of Science\\Kyoto University\\Kitashirakawa-Oiwake cho\\Sakyo-ku\\Kyoto\\606-8502\\Japan}
\email{ahigashi@math.kyoto-u.ac.jp} 
\author[A.\,M.\,Kasprzyk]{Alexander~Kasprzyk}
\address{Department of Mathematics\\Imperial College London\\London\ SW7 2AZ\\United Kingdom}
\email{a.m.kasprzyk@imperial.ac.uk} 
\title{Ehrhart polynomial roots of reflexive polytopes}
\maketitle
\begin{abstract}
Recent work has focused on the roots $z\in\C$ of the Ehrhart polynomial of a lattice polytope~$P$. The case when $\Re{z}=-1/2$ is of particular interest: these polytopes satisfy Golyshev's ``canonical line hypothesis''. We characterise such polytopes when $\dim{P}\leq 7$. We also consider the ``half-strip condition'', where all roots $z$ satisfy $-\dim{P}/2\leq\Re{z}\leq \dim{P}/2-1$, and show that this holds for any reflexive polytope with $\dim{P}\leq 5$. We give an example of a $10$-dimensional reflexive polytope which violates the half-strip condition, thus improving on an example by Ohsugi--Shibata in dimension~$34$.
\end{abstract}
\section{Introduction}\label{sec:intro}
Let $P\subset\Z^d\otimes_\Z\Q$ be a convex lattice polytope of dimension $d$. The \emph{Ehrhart polynomial}~\cite{Ehr62} $L_P$ counts the number of lattice points in successive dilations of $P$, i.e.~$L_P(m)=\abs{mP\cap\Z^d}$ for all $m\in\Z_{\geq 0}$, and is a polynomial of degree $d$. Stanley~\cite{Stan80} showed that the corresponding generating series, the \emph{Ehrhart series}, can be written as a rational function with numerator a degree $d$ polynomial whose coefficients define the so-called \emph{$\delta$-vector} (or \emph{$h^*$-vector}) of $P$:
\[
\Ehr_P(t)=\frac{\delta_0+\delta_1t+\ldots+\delta_dt^d}{(1-t)^{d+1}}=\sum_{m\geq 0}L_P(m)t^m.
\]
Starting with a $\delta$-vector one can easily recover the Ehrhart polynomial:

\begin{lemma}\label{lem:main_lemma}
Let $P$ be a $d$-dimensional convex lattice polytope with $\delta$-vector $(\delta_0,\delta_1,\ldots,\delta_d)$. Then
\[
L_P(m)=\sum_{j=0}^{d} \delta_j{d+m-j\choose d}.
\]
\end{lemma}

\noindent
In general, combinatorial interpretations for the coefficients $\delta_i$ of the $\delta$-vector are not known, however the work of Ehrhart~\cite{Ehr67} tells us that:
\begin{enumerate}
\item
$\delta_0=1$;
\item\label{item:delta_1}
$\delta_1=\abs{P\cap\Z^d}-d-1$;
\item
$\delta_d=\abs{\intr{P}\cap\Z^d}$, where $\intr{P}=P\setminus\bdry{P}$ denotes the (strict) interior of $P$;
\item
$\delta_0+\ldots+\delta_d=d!\vol{P}$, where $\vol{P}$ denotes the (non-normalised) volume of $P$.
\end{enumerate}
Hibi's Lower Bound Theorem~\cite{HibiLBT} states that if $\abs{\intr{P}\cap\Z^d}>0$ then $\delta_1\leq\delta_i$ for each $2\leq i\leq d-1$. In particular, combined with~\eqref{item:delta_1} we see that the $\delta_i$ are positive.

A convex lattice polytope $P$ is called \emph{reflexive} if the \emph{dual} (or \emph{polar}) polyhedron
\[
\dual{P}:=\{u\in\Z^d\otimes_\Z\Q\mid\left<u,v\right>\geq -1\text{ for all }v\in P\}
\]
is also a lattice polytope. If $P$ is reflexive then $\dual{P}$ is also reflexive, and $\intr{P}\cap\Z^d=\{0\}$. Reflexive polytopes are of particular importance in toric geometry: they correspond to Gorenstein toric Fano varieties and are a key combinatorial tool, as introduced by Batyrev~\cite{Bat94}, for constructing topologically mirror-symmetric pairs of Calabi--Yau varieties. Reflexive polytopes were characterised by Hibi~\cite{Hib91} as being those polytopes (up to lattice translation) with palindromic $\delta$-vectors, i.e.~$\delta_i=\delta_{d-i}$ for each $0\leq i\leq d$. A special class of reflexive polytopes are the \emph{smooth Fano polytopes}. These are simplicial reflexive polytopes $P$ such that, for each facet $F$, the vertices $\V{F}$ of $F$ generate the underlying lattice $\Z^d$. They correspond to the smooth toric Fano varieties. For a summary of the various equivalences between lattice polytopes and toric Fano varieties, see~\cite{KN12}.

Several recent results have concentrated on the roots of the Ehrhart function $L_P$, regarded as a polynomial over $\C$. These results are inspired in part by Rodriguez-Villegas' study~\cite{R-V02} of Hilbert polynomials all of whose roots $z\in\C$ lie on a line $\Re{z}=-a/2$, and the connection when $a=1$ with the Riemann zeta function. Braun~\cite{Bra08} has shown that the roots of $L_P$ lie inside a disc centred at $-1/2$ with radius $d(d-1/2)$. Beck--de~Loera--Develin--Pfeifle--Stanley~\cite{BLDPS05} and Braun~\cite{BD08} have also shown that the roots $z$ lie in a strip
\begin{equation}\tag{S}\label{eq:S}
-d\leq\Re{z}\leq d-1
\end{equation}
when $d\leq 5$, or for arbitrary $d$ when $\Im{z}=0$. An example~\cite{H12} that fails to satisfy the \emph{strip condition}~\eqref{eq:S} is known in dimension~$15$.

When $P$ is a reflexive polytope, Macdonald's Reciprocity Theorem~\cite{Mac71} gives that $L_P(-m-1)=(-1)^dL_P(m)$, and so the roots of $L_P$ are symmetrically distributed with respect to the line $\Re{z}=-1/2$. In fact:

\begin{lemma}
The lattice polytope $P$ is reflexive (up to lattice translation) if and only if
\begin{equation}\label{eq:reflexive_characterisation}
\sum_{i=1}^dz_i=-\frac{d}{2}
\end{equation}
where $z_1,\ldots,z_d\in\C$ are the roots of $L_P$.
\end{lemma}

\begin{proof}
Write $L_P(z)=c_dz^d+c_{d-1}z^{d-1}+\ldots+c_0$. By~\cite{Ehr67} we have that $c_d=\vol{P}$, $c_{d-1}=1/2\vol{\bdry{P}}$, and $c_0=1$. Hence:
\[
-\sum_{i=1}^dz_i=\frac{c_{d-1}}{\vol{P}}=\frac{1}{2}\cdot\frac{\vol{\bdry{P}}}{\vol{P}}
\]
But $P$ is reflexive (up to lattice translation) if and only if $d\vol{P}=\vol{\bdry{P}}$~\cite{Bat94}. Hence~\eqref{eq:reflexive_characterisation} characterises when a polytope is reflexive.
\end{proof}

Bey--Henk--Wills~\cite{BHW07} proved that if
\begin{equation}\tag{CL}\label{eq:CL}
\Re{z}=-\frac{1}{2}
\end{equation}
for all roots $z$ of $L_P$ then $P$ is a reflexive polytope (after possible translation by a lattice vector). Golyshev~\cite{Gol09} proposed the study of Fano varieties satisfying the so-called \emph{canonical line hypothesis}; in the context of lattice polytopes this is equivalent to the roots satisfying condition~\eqref{eq:CL}. He conjectured that every $d$-dimensional smooth Fano polytope with $d\leq 5$ satisfies the canonical line hypothesis. This was proved in~\cite{HK11}, along with an example of a smooth Fano polytope in dimension~$6$ failing to satisfy~\eqref{eq:CL}.

With the above results in mind, we introduce the following terminology:

\begin{definition}
We say that a lattice polytope $P$ is a \emph{CL-polytope} if the roots $z\in\C$ of the Ehrhart polynomial $L_P$ satisfy~\eqref{eq:CL}. We say that $P$ is \emph{real} if the imaginary part $\Im{z}=0$ for all roots $z$ of $L_P$.
\end{definition}

\begin{prop}[\protect{\cite[Proposition~1.8]{BHW07}}]\label{prop:CL_volume}
Let $P$ be a $d$-dimensional CL-polytope. Then $\vol{P}\leq 2^d$.
\end{prop}

\begin{proof}
First we consider the case when $d=2k$. Let $z_i\in\C$ denote the roots of the Ehrhart polynomial $L_P$. Then
\[
\prod_{j=1}^dz_j=\frac{1}{\vol{P}}.
\]
We can assume that $z_{2j-1}=-1/2-\alpha_ji$ and $z_{2j}=-1/2+\alpha_ji$, where $\alpha_j\in\R$, for each $1\leq j\leq k$. Hence
\[
\frac{1}{\vol{P}}=\prod_{j=1}^kz_{2j-1}z_{2j}=\prod_{j=1}^k\left(\frac{1}{4}+\alpha_j^2\right)\geq\frac{1}{4^k}= \frac{1}{2^d}.
\]
The case when $d=2k+1$ is similar.
\end{proof}

\noindent
In~\S\S\ref{sec:dim23}--\ref{sec:dim5} we characterise when a reflexive polytope $P$ with $d\leq 5$ is a CL-polytope:

\begin{thm}[Canonical line hypothesis]\label{thm:canonical_line}
Let $P$ be a $d$-dimensional reflexive polytope.
\begin{enumerate}
\item[$d=2$:]
$P$ is a CL-polytope if and only if $\vol{P}\leq 4$, or equivalently if and only if $\abs{P\cap\Z^2}\leq 9$.
\item[$d=3$:]
$P$ is a CL-polytope if and only if $\vol{P}\leq 8$, or equivalently if and only if $\abs{P\cap\Z^3}\leq 27$.
\item[$d=4$:]
$P$ is a CL-polytope if and only if either $\vol{P}=16$ and $\abs{P\cap\Z^4}=81$, or
\[
5\vol{P}<\abs{P\cap\Z^4}-1,\ 2\abs{P\cap\Z^4}\le 9(\vol{P} + 2),\text{ and }\left(\abs{P\cap\Z^4}-1-4\vol{P}\right)^2\ge 16\vol{P}.
\]
\item[$d=5$:]
$P$ is a CL-polytope if and only if either $\vol{P}=32$ and $\abs{P\cap\Z^5}=243$, or
\[
15 \vol{P}<2\left(\abs{P\cap\Z^5}-3\right),\ 4\abs{P\cap\Z^5}\leq 27(\vol{P}+4),\text{ and }\left(\abs{P\cap\Z^5}-3-6\vol{P}\right)^2\geq 72\vol{P}.
\]
\end{enumerate}
\end{thm}

\noindent
A characterisation when $d=6$ and~$7$ is given in~\S\ref{sec:dim6_7_CL}, Theorems~\ref{thm:dim6_conditions} and~\ref{thm:dim7_conditions}. As a consequence we have:

\begin{cor}\label{cor:cube}
Let $P$ be a $d$-dimensional CL-polytope with $d\leq 7$. Then $\abs{P\cap\Z^d}\leq 3^d$.
\end{cor}

\noindent
We conjecture that Corollary~\ref{cor:cube} holds for all $d$; if true the result would be sharp, since equality is achieved by the $d$-dimensional cube $\{-1,1\}^d$ corresponding to the anticanonical polytope of $(\Proj^1)^d$. In~\S\S\ref{sec:dim23}--\ref{sec:dim5} we also give characterisations for the real case when $d\leq 5$. These characterisations differ from the case of CL-polytopes in Theorem~\ref{thm:canonical_line} simply by flipping both of the (equivalent) inequalities when $d=2$ or~$3$, and by flipping the first inequality when $d=4$ or~$5$. In~\S\ref{sec:dim6_7_real}, Theorems~\ref{thm:dim6_real_conditions} and~\ref{thm:dim7_real_conditions}, we give characterisations when $d=6$ and $7$; again these differ from the case for CL-polytopes by flipping certain inequalities.

It was conjectured in~\cite{MHNOH11} that for any reflexive polytope $P$, the roots $z$ of $L_P$ satisfy the \emph{half-strip condition}
\begin{equation}\tag{HS}\label{eq:HS}
-\frac{d}{2}\leq\Re{z}\leq\frac{d}{2}-1.
\end{equation}
In~\S\ref{sec:canonical_strip} we prove the following:

\begin{thm}\label{thm:reflexive_strip}
Let $P$ be a $d$-dimensional reflexive polytope and let $z\in\C$ be a root of $L_P$. Inequality~\eqref{eq:HS} holds when $d\leq 5$, or for arbitrary $d$ when $P$ is real.
\end{thm}

\noindent
This result halves the bounds of~\eqref{eq:S} in the case of reflexive polytopes. Ohsugi--Shibata~\cite{OS12} have found a $34$-dimensional reflexive polytope failing to satisfy~\eqref{eq:HS}. In~\S\ref{subsec:find_P_from_delta} we give a general method for determining whether a palindromic $\delta$-vector with $\delta_1=1$ arises from a lattice polytope, and use this in~\S\ref{subsec:dim_10_example} to give a $10$-dimensional example that fails to satisfy~\eqref{eq:HS}. This corresponds to a terminal Gorenstein fake weighted projective space. It seems probable that Theorem~\ref{thm:reflexive_strip} also holds when $d=6$ and~$7$, and possible that $d=10$ is the smallest dimension in which~\eqref{eq:HS} fails, although we do not prove this.

\subsection*{A hierarchy of hypotheses}
Golyshev~\cite{Gol09} introduced two additional bounds on the roots $z\in\C$ of the Ehrhart polynomial $L_P$, which he called the \emph{canonical strip hypothesis} and the \emph{narrowed canonical strip hypothesis}. These correspond, respectively, to:
\begin{align}
\tag{CS}\label{eq:CS}
-1&<\Re{z}<0\\
\tag{NCS}\label{eq:NCS}
-\frac{d}{d+1}&\leq\Re{z}\leq\frac{d}{d+1}-1.
\end{align}

\noindent
By a slight modification of the proof of Proposition~\ref{prop:CL_volume} above, we obtain:

\begin{prop}\label{prop:real_reflexive_CS_volume}
Let $P$ be a $d$-dimensional real reflexive polytope such that the roots $z\in\C$ of $L_P$ satisfy~\eqref{eq:CS}. Then $\vol{P}\geq 2^d$.
\end{prop}

We have a hierarchy of implications
\begin{center}
\eqref{eq:CL}\ $\Longrightarrow$\ \eqref{eq:NCS}\ $\Longrightarrow$\ \eqref{eq:CS}\ $\Longrightarrow$\ \eqref{eq:HS}\ $\Longrightarrow$\ \eqref{eq:S}.
\end{center}
These hypotheses provide meaningful ways of partitioning the space of reflexive polytopes (or, more generally, Fano polytopes) by their $\delta$-vectors. In higher dimensions, where the number of reflexive polytopes is vast, this becomes an essential tool for studying their classification.

\section{Dimensions two and three}\label{sec:dim23}
We begin with a general observation. Let $L_1:=\{z\in\C\mid\Re{z}=-1/2\}$ and $L_2:=\{z\in\C\mid\Im{z}=0\}$ be lines on the complex plane. By Macdonald's Reciprocity Theorem we see that any reflexive polytope $P$ of dimension $d$ satisfies
\[
L_P(-m-1)=(-1)^d L_P(m).
\]
In particular the roots of $L_P$ are distributed symmetrically with respect to the lines $L_1$ and, via complex conjugation, $L_2$. When $\dim{P}=2$ or $3$ we obtain the following two results (cf.~\cite[Proposition~1.9(i)]{BHW07}):

\begin{prop}\label{prop:dim_2_dicotomy}
Let $P$ be a $2$-dimensional reflexive polytope with $\delta$-vector $(1,\delta_1,1)$. Then either
\begin{enumerate}
\item\label{cond:dim_2_Riemannian}
the roots of $L_P$ are $\{-1/2\pm bi\}$, $b\in\R$, i.e.~$P$ is a CL-polytope, or
\item\label{cond:dim_2_real}
the roots of $L_P$ are $\{a,-1-a\}$, $a\in\R$, i.e.~$P$ is real.
\end{enumerate}
Furthermore, the following are equivalent:

\vspace{0.4em}
\begin{tabular}{lp{2.5cm}p{2.5cm}cl}
\emph{\phantom{i}(ia)} $P$ is a CL-polytope;&
\emph{\phantom{i}(ib)} $\delta_1\leq 6$;&
\emph{\phantom{i}(ic)} $\vol{P}\leq 4$;&
and&
\emph{\phantom{i}(id)} $\abs{P\cap\Z^2}\leq 9$.\\
\emph{(iia)} $P$ is real;&
\emph{(iib)} $\delta_1\geq 6$;&
\emph{(iic)} $\vol{P}\geq 4$;&
and&
\emph{(iid)} $\abs{P\cap\Z^2}\geq 9$.
\end{tabular}
\vspace{0.4em}

\noindent
In addition, in case~\eqref{cond:dim_2_real} one has that $-1<a<0$, i.e.\ $P$ satisfies~\eqref{eq:CS}.
\end{prop}

\begin{proof}
We only give a proof of the last assertion. It follows from Lemma~\ref{lem:main_lemma} that
\[
L_P(z)=\frac{1}{2}\left(\delta_1+2\right)z^2+\frac{1}{2}\left(\delta_1+2\right)z+1.
\]
The roots of $L_P$ are
\[
-\frac{1}{2}\pm\frac{1}{2}\sqrt{\frac{\delta_1-6}{\delta_1+2}}.
\]
Since we have that $\delta_1\geq 6$ in case~\eqref{cond:dim_2_real}, the roots of $L_P(z)$ satisfy~\eqref{eq:CS}.
\end{proof}

\begin{prop}\label{prop:dim_3_dicotomy}
Let $P$ be a $3$-dimensional reflexive polytope with $\delta$-vector $(1,\delta_1,\delta_1,1)$. Then either
\begin{enumerate}
\item\label{cond:dim_3_Riemannian}
the roots of $L_P$ are $\{-1/2,-1/2\pm bi\}$, $b\in\R$, i.e.~$P$ is a CL-polytope, or
\item\label{cond:dim_3_real}
the roots of $L_P$ are $\{-1/2,a,-1-a\}$, $a\in\R$, i.e.~$P$ is real.
\end{enumerate}
Furthermore, the following are equivalent:

\vspace{0.4em}
\begin{tabular}{lp{2.5cm}p{2.5cm}cl}
\emph{\phantom{i}(ia)} $P$ is a CL-polytope;&
\emph{\phantom{i}(ib)} $\delta_1\leq 23$;&
\emph{\phantom{i}(ic)} $\vol{P}\leq 8$;&
and&
\emph{\phantom{i}(id)} $\abs{P\cap\Z^3}\leq 27$.\\
\emph{(iia)} $P$ is real;&
\emph{(iib)} $\delta_1\geq 23$;&
\emph{(iic)} $\vol{P}\geq 8$;&
and&
\emph{(iid)} $\abs{P\cap\Z^3}\geq 27$.
\end{tabular}
\vspace{0.4em}

\noindent
In addition, in case~\eqref{cond:dim_3_real} one has that $-1<a<0$, i.e.\ $P$ satisfies~\eqref{eq:CS}.
\end{prop}

\begin{proof}
It follows from Lemma~\ref{lem:main_lemma} that
\[
L_P(z)=\frac{1}{3}\left(\delta_1+1\right)z^3+\frac{1}{2}\left(\delta_1+1\right)z^2+\frac{1}{6}\left(\delta_1+13\right)z+1.
\]
On the other hand, $P$ is a CL-polytope if and only if there exists $b\in\R$ such that
\[
L_P(z)=\frac{2+2\delta_1}{3!}\left(z+\frac{1}{2}\right)\left(z+\frac{1}{2}-bi\right)\left(z+\frac{1}{2}+bi\right).
\]
By comparing the constant term in the two expressions, we see that $P$ is a CL-polytope if and only if there exists $b\in\R$ such that
\[
4(\delta_1+1)b^2+\delta_1-23=0.
\]
Consequently, $P$ is a CL-polytope if and only if $\delta_1\leq 23$.

A reflexive polytope $P$ is real if and only if there exists $a\in\R$ such that
\[
L_P(z)=\frac{2+2\delta_1}{3!}\left(z+\frac{1}{2}\right)\left(z-a\right)\left(z+1+a\right).
\]
Again, comparing the constant term gives $(1+\delta_1)a^2+(1+\delta_1)a+6=0$, hence $P$ is real if and only if $\delta_1\geq 23$. Moreover, since the solutions to this quadratic are given by
\[
-\frac{1}{2}\pm\frac{1}{2}\sqrt{\frac{\delta_1-23}{\delta_1+1}},
\]
we see that $-1<a<0$ and~\eqref{eq:CS} is satisfied.
\end{proof}

\section{Dimension four}\label{sec:dim4}
Let $P$ be a $4$-dimensional reflexive polytope with $\delta$-vector $(1,\delta_1,\delta_2,\delta_1,1)$. The roots of $L_P$ fall into four possible cases:
\begin{enumerate}[label=(\alph*), ref=\alph*]
\item\label{item:dim4_case2}
The roots of $L_P$ are $\{-1/2\pm b_1i,-1/2\pm b_2i\}$, where $b_1,b_2\in\R$. In this case $P$ is a CL-polytope.
\item\label{item:dim4_case3}
The roots of $L_P$ are $\{-1/2\pm a_1,-1/2\pm a_2\}$, where $a_1,a_2\in\R$. In this case $P$ is real.
\item\label{item:dim4_case4}
The roots of $L_P$ are $\{-1/2\pm a,-1/2\pm bi\}$, where $a,b\in\R$.
\item\label{item:dim4_case1}
The roots of $L_P$ are $\{-1/2+a\pm bi,-1/2-a\pm bi\}$, where $a,b\in\R\setminus\{0\}$.
\end{enumerate}

We shall require the following trivial lemma.
Although the proof is obvious, we note that one approach that generalises well to higher degree is to employ Descartes' rule of signs.

\begin{lemma}\label{lem:quadratic}
Let $A,B,C \in \R$ be real numbers, $A>0$. Then all solutions of the equation
\[
Az^2+Bz+C=0
\]
are non-negative real numbers if and only if either $B=C=0$, or $B<0$, $C \geq 0$, and the discriminant $\Disc\geq 0$.
\end{lemma}

\subsection{Four-dimensional CL-polytopes}
\begin{thm}\label{thm:dim_4_conditions}
Let $P$ be a $4$-dimensional reflexive polytope with $\delta$-vector $(1,\delta_1,\delta_2,\delta_1,1)$. Then $P$ is a CL-polytope (i.e.~we are in case~\eqref{item:dim4_case2} above) if and only if either
\begin{enumerate}
\item\label{item:dim_4_conditions_1}
$\delta_1=76$ and $\delta_2=230$, or
\item\label{item:dim_4_conditions_2}
$5\delta_2<14\delta_1+86$, $10\delta_1\leq 3\delta_2+70$, and $17(\delta_1+4\delta_2-15)^2\leq(17\delta_1+49)^2+(17\delta_2-94)^2$.
\end{enumerate}
In particular, if $P$ is a CL-polytope then $\delta_1\leq 76$ and $\delta_2\leq 230$.
\end{thm}
\begin{proof}
From Lemma~\ref{lem:main_lemma} we obtain
\begin{equation}\label{eq:dim_4_Ehrhart}
 L_P(z)=(2+2\delta_1+\delta_2){z\choose 4}+(4+4\delta_1+2\delta_2){z\choose 3}
+(6+3\delta_1+\delta_2){z\choose 2}+(4+\delta_1){z\choose 1}+1.
\end{equation}
Substituting $z=-1/2+\beta i$ in~\eqref{eq:dim_4_Ehrhart} and multiplying through by $4!/2$ gives
\[
G(\beta):=\left(1+\delta_1+\frac{1}{2}\delta_2\right)\beta^4-\frac{1}{2}\left(43+7\delta_1-\frac{5}{2}\delta_2\right)\beta^2+\frac{1}{16}\left(105-15\delta_1+\frac{9}{2}\delta_2\right),
\]
and $G(\beta)=0$ if and only if $L_P(-1/2+\beta i)=0$. Regarding $G$ as a quadratic in $\beta^2$ we obtain the discriminant:
\[
\Disc=17\left(\delta_1+\frac{49}{17}\right)^2+17\left(\delta_2-\frac{94}{17}\right)^2-(\delta_1+4\delta_2-15)^2.
\]
The two cases follow from Lemma~\ref{lem:quadratic}. The first two inequalities in case~\eqref{item:dim_4_conditions_2} give $\delta_1<76$ and $\delta_2<230$; combining these bounds with case~\eqref{item:dim_4_conditions_1} we see that $\delta_1\leq 76$ and $\delta_2\leq 230$.
\end{proof}

The discriminant $\Delta$ gives rise to the parabola $17(\delta_1+4\delta_2-15)^2=(17\delta_1+49)^2+(17\delta_2-94)^2$ with focus at $(\delta_1,\delta_2)=(-49/17,94/17)$. The tangent at the point $(76,230)$ is given by $10\delta_1=3\delta_2+70$. Together these two equations (or, more accurately, the corresponding two inequalities in Theorem~\ref{thm:dim_4_conditions}\eqref{item:dim_4_conditions_2}) describe three regions in the positive quadrant. The inequality $5\delta_2<14\delta_1+86$ specifies which of these three regions contains the CL-polygons. This is illustrated in Figure~\ref{fig:3d_parabola_and_tangent}. There is clearly a choice for this inequality.

\begin{example}
Case~\eqref{item:dim_4_conditions_1} in Theorem~\ref{thm:dim_4_conditions} can certainly occur: the $4$-dimensional cube $\{-1,1\}^4$ is a reflexive polytope with $\delta$-vector $(1,76,230,76,1)$. Notice that this is not the only polytope with this $\delta$-vector: a second example is given in Example~\ref{eg:P12333}. In general it would be an interesting problem to classify all polytopes $P$ with $\delta$-vector equal to the $d$-dimensional cube $\{-1,1\}^d$, that is, with Ehrhart polynomial $L_{P}(m)=(2m+1)^d$.
\end{example}

Theorem~\ref{thm:dim_4_conditions} tells us that $\delta_i\leq\delta_i'$, for each $0\leq i\leq 4$, where $\delta'$ is the $\delta$-vector for the $4$-dimensional cube $\{-1,1\}^4$. In particular we have that $\vol{P}\leq 2^4$ and $\abs{P\cap\Z^4}\leq 3^4$.

\begin{example}\label{eg:delta_d_cube}
We shall calculate the $\delta$-vector for the $d$-dimensional cube $\{-1,1\}^d$. In general let $P$ and $Q$ be lattice polytopes such that $L_P(m)=L_Q(2m)$ for all $m\in\Z_{\geq 0}$, and let $(\delta_0,\delta_1,\ldots,\delta_d)$ be the $\delta$-vector of $Q$. Then
\[
\frac{\Ehr_Q(t)+\Ehr_Q(-t)}{2}=\sum_{m\geq 0}L_Q(2m)t^{2m}=\Ehr_P(t^2).
\]
Now
\begin{align*}
\frac{\Ehr_Q(t)+\Ehr_Q(-t)}{2}&=\frac{(1+t)^{d+1}(\delta_0+\delta_1t+\ldots+\delta_dt^d)+(1-t)^{d+1}(\delta_0-\delta_1t+\ldots+(-1)^d\delta_dt^d)}{2(1-t^2)^{d+1}}\\
&=\frac{\sum_{i=0}^{2d+1}t^i\sum_{j=0}^d\left({d+1\choose i-j}+(-1)^i{d+1\choose i-j}\right)\delta_j}{2(1-t^2)^{d+1}}\\
&=\frac{\sum_{i=0}^dt^{2i}\sum_{j=0}^d{d+1\choose 2i-j}\delta_j}{(1-t^2)^{d+1}},
\end{align*}
hence the $\delta$-vector of $P$ is given by
\[
\delta^P_i=\sum_{j=0}^d{d+1\choose 2i-j}\delta_j,\qquad\text{ for each }i\in\{0,\ldots,d\}.
\]
Now let $Q=\{0,1\}^d$ be the $d$-dimensional cube of unit volume, and $P=\{-1,1\}^d$, so that$L_P(m)=L_Q(2m)$. It is well-known that the $\delta$-vector of $Q$ can be expressed in terms of the Eulerian numbers, $\delta_i=A(d,i)$, hence
\[
\delta^P_i=\sum_{j=0}^d{d+1\choose 2i-j}A(d,j),\qquad\text{ for each }i\in\{0,\ldots,d\}.
\]
\end{example}

\begin{example}\label{eg:P12333}
Let $P=\sconv{(1,0,0,0),(0,1,0,0),(0,0,1,0),(0,0,0,1),(-2,-3,-3,-3)}$ be the polytope corresponding to weighted projective space $X=\Proj(1,2,3,3,3)$. The dual polytope
\[
\dual{P}=\sconv{(-1,-1,-1,-1),(5,-1,-1,-1),(-1,3,-1,-1),(-1,-1,3,-1),(-1,-1,-1,3)}
\]
corresponding to the anticanonical divisor $-K_X$ has:
\begin{align*}
4!\vol{\dual{P}}&=(-K_X)^4=\frac{(1+2+3+3+3)^4}{1\cdot 2\cdot 3\cdot 3\cdot 3}=384,\\
\abs{\dual{P}\cap\Z^4}&=h^0(X_P,-K_X)=\#\{\text{monomials of weighted degree }1+2+3+3+3=12\}=81.
\end{align*}
Equivalently, the value of $h^0(X_P,-K_X)$ is equal to the coefficient of $t^{12}$ in the Taylor expansion
\[
\frac{1}{(1-t)(1-t^2)(1-t^3)^3}=1+t+2t^2+5t^3+\ldots+81t^{12}+\ldots
\]
We have that $\delta_1=76$ and $\delta_2=230$, and $P$ is both a CL-polytope and a real polytope.
\end{example}

\begin{cor}\label{cor:dim_4_conditions}
Let $P$ be a $4$-dimensional reflexive polytope. Then $P$ is a CL-polytope (i.e.~we are in case~\eqref{item:dim4_case2} above) if and only if either
\begin{enumerate}
\item
$\vol{P}=16$ and $\abs{P\cap\Z^4}=81$, or
\item\label{item:dim_4_conditions}
$5\vol{P}<\abs{P\cap\Z^4}-1$, $2\abs{P\cap\Z^4}\le 9(\vol{P} + 2)$, and $\left(\abs{P\cap\Z^4}-1-4\vol{P}\right)^2\ge 16\vol{P}$.
\end{enumerate}
\end{cor}

\begin{proof}
Simply make the substitutions $2+2\delta_1+\delta_2=4!\vol{P}$ and $\delta_1=\abs{P\cap\Z^4}-5$ in Theorem~\ref{thm:dim_4_conditions}.
\end{proof}

\noindent
The second and third inequalities in Corollary~\ref{cor:dim_4_conditions}\eqref{item:dim_4_conditions} also appear in the work of Bey--Henk--Wills~\cite[Proposition~1.9(ii)]{BHW07}, however they overlook the first inequality. As noted above, this inequality is necessary in order to specify which of the regions we are interested in, although the precise form this inequality takes is a matter of choice. In Theorem~\ref{thm:dim_4_real} we will show that flipping this inequality corresponds to selecting the regions containing the real reflexive polytopes.

\subsection{Four-dimensional real reflexive polytopes}\label{sec:dim_4_real_reflexive}
\begin{thm}\label{thm:dim_4_real}
Let $P$ be a $4$-dimensional reflexive polytope with $\delta$-vector $(1,\delta_1,\delta_2,\delta_1,1)$. Then $P$ is a real polytope (i.e.~we are in case~\eqref{item:dim4_case3} above) if and only if either
\begin{enumerate}
\item\label{item:dim_4_real_1}
$\delta_1=76$ and $\delta_2=230$, or
\item\label{item:dim_4_real_2}
$5\delta_2>14\delta_1+86$, $10\delta_1\leq 3\delta_2 + 70$, and $17(\delta_1+4\delta_2-15)^2\leq(17\delta_1+49)^2+(17\delta_2-94)^2$.
\end{enumerate}
\end{thm}

\begin{proof}
Let $z=-1/2+\alpha$, $\alpha\in\R$, be a root of $L_P$. Substituting this into~\eqref{eq:dim_4_Ehrhart} gives
\begin{equation}\label{eq:dim4_alpha}
\frac{\alpha^4}{12}\left(1+\delta_1+\frac{1}{2}\delta_2\right) + \frac{\alpha^2}{24}\left(43+7\delta_1-\frac{5}{2}\delta_2\right)+\frac{1}{64}\left(35-5\delta_1+\frac{3}{2}\delta_2\right) = 0.
\end{equation}
Regarding this as a quadratic in $\alpha^2$, the result follows from Lemma~\ref{lem:quadratic}.
\end{proof}

\begin{example}
Let $Q=\sconv{(-1,2),(-1,-1),(2,-1)}$ with $L_Q(m)=1/2(3m+1)(3m+2)$. Consider the direct product $P=Q\times Q$. This is a $4$-dimensional polytope with Ehrhart polynomial
\[
L_P(m)=L_Q(m)^2=\frac{1}{4}(3m+1)^2(3m+2)^2.
\]
Hence $P$ has $\delta$-vector $(1,95,294,95,1)$ and this gives equality in the third expression in Theorem~\ref{thm:dim_4_real}\eqref{item:dim_4_real_2}, i.e.\ $(95,294)$ is an integer point on the parabola defined by the discriminant $\Delta$.
\end{example}

\begin{example}
We investigate which $\delta$-vectors lie on the parabola $17(\delta_1+4\delta_2-15)^2=(17\delta_1+49)^2+(17\delta_2-94)^2$. Since we already know an integer point, $\delta_1=76$, $\delta_2=230$, by the ``slope method'' we can easily parameterise the rational points on the curve:
\[
\delta_1=\frac{4(19\gamma^2-98\gamma+124)}{(\gamma-4)^2},\qquad\delta_2=\frac{2(223\gamma^2-1280\gamma+1840)}{(\gamma-4)^2},\qquad\text{ where }\gamma\in\Q.
\]
Interpreting the first of these equations as a quadratic in $\gamma$, and restricting to $\delta_1\in\Z$, we see that $\delta_1+5$ is a square; this is equivalent to saying that $\abs{P\cap\Z^4}$ is a square. Setting $\delta_1+5=\abs{P\cap\Z^4}=N^2$ for some $N\in\Z_{\geq 1}$ and solving for $\gamma$, we obtain:
\[
\gamma=4-\frac{12}{9\pm N}.
\]
The second equation gives that $\delta_2-5$ is a square, and setting $\delta_2-5=M^2$ for some $M\in\Z_{\geq 1}$ we obtain:
\[
\gamma=4-\frac{24}{21\pm M}.
\]
Equating these two expressions for $\gamma$, and remembering that $\abs{P\cap\Z^4}\geq6$, we find that
\[
(\delta_1,\delta_2)=(N^2-5,(2N\pm3)^2+5),\qquad\text{ for }N\geq 3.
\]
By consulting the Kreuzer--Skarke classification~\cite{KS00} we see that the cases with $\delta_2=(2N+3)^2+5$ never occur (this corresponds to the upper branch of the parabola); the reflexive polytopes lie on the bottom branch with $\delta_2=(2N-3)^2+5$, for each $N\in\{3,\ldots,13\}$. When $N>13$ there are no matching $\delta$-vectors. The occurring $\delta$-vectors are recorded in Table~\ref{tab:parabola_deltas}.
\end{example}

\begin{table}[htdp]
\caption{The $\delta$-vectors of the $4$-dimensional reflexive polytopes lying on the parabola $17(\delta_1+4\delta_2-15)^2=(17\delta_1+49)^2+(17\delta_2-94)^2$. Those with $3\leq N\leq 9$ correspond to CL-polytopes; those with $9\leq N\leq 13$ correspond to real polytopes.}\label{tab:parabola_deltas}
\centering
\begin{tabular}{rccccccccccc}
\toprule
$N$&$3$&$4$&$5$&$6$&$7$&$8$&$9$&$10$&$11$&$12$&$13$\\
\midrule
$\delta_1$&$4$&$11$&$20$&$31$&$44$&$59$&$76$&$95$&$116$&$139$&$164$\\
$\delta_2$&$14$&$30$&$54$&$86$&$126$&$174$&$230$&$294$&$366$&$446$&$534$\\
\bottomrule
\end{tabular}
\end{table}

\begin{cor}\label{cor:dim_4_conditions_real}
Let $P$ be a $4$-dimensional reflexive polytope. Then $P$ is a real polytope (i.e.~we are in case~\eqref{item:dim4_case3} above) if and only if either
\begin{enumerate}
\item
$\vol{P}=16$ and $\abs{P\cap\Z^4}=81$, or
\item\label{item:dim_4_conditions}
$5\vol{P}>\abs{P\cap\Z^4}-1$, $2\abs{P\cap\Z^4}\le 9(\vol{P} + 2)$, and $\left(\abs{P\cap\Z^4}-1-4\vol{P}\right)^2\ge 16\vol{P}$.
\end{enumerate}
\end{cor}

\begin{example}\label{ex:essential}
Consider the simplex
\[
P=\sconv{(4,-1,-1,-1),(-1,4,-1,-1),(-1,-1,4,-1),(-1,-1,-1,4),(-1,-1,-1,-1)}
\]
corresponding to the anticanonical polytope of $X=\Proj^4$. This has
\begin{align*}
4!\vol{P}&=(-K_X)^4=5^4,\\
\abs{P\cap\Z^4}&=h^0(X,-K_X)={{2\cdot 5-1}\choose{5}}=126,
\end{align*}
satisfying Corollary~\ref{cor:dim_4_conditions_real}\eqref{item:dim_4_conditions}. The Ehrhart polynomial is
\[
L_P(m)=\frac{1}{4!}\prod_{k=1}^4(5m+k).
\]
\end{example}

\begin{prop}\label{dim4_real_volume}
Let $P$ be a $4$-dimensional real reflexive polytope. Then $\vol{P}\geq 3$.
\end{prop}

\begin{proof}
By Theorem~\ref{thm:dim_4_real}, $P$ is real if and only if $\delta_1$, $\delta_2$ satisfy either conditions~\eqref{item:dim_4_real_1} or~\eqref{item:dim_4_real_2}. In case~\eqref{item:dim_4_real_1}, since $\vol{P}=(2+2\delta_1+\delta_2)/24$, there is nothing to prove. In case~\eqref{item:dim_4_real_2}, by the inequality $17(\delta_1+4\delta_2-15)^2\leq(17\delta_1+49)^2+(17\delta_2-94)^2$ we have $16\delta_1^2-8(\delta_2-16)\delta_1+\delta_2^2-68\delta_2+436 \geq 0$, that is:
\begin{align}\label{eq:condition}
\delta_1\geq\frac{(\delta_2-16)+6\sqrt{\delta_2-5}}{4},\qquad\text{ or }\qquad 1\leq\delta_1\leq \frac{(\delta_2-16)-6\sqrt{\delta_2-5}}{4}.
\end{align}
Moreover, by the inequality $5\delta_2>14\delta_1+86$, we have that $\delta_1<(5\delta_2-86)/14$.

When the first condition in~\eqref{eq:condition} is satisfied, and recalling that $\delta_1 \leq \delta_2$, we see that
\[
\frac{(\delta_2-16)+6\sqrt{\delta_2-5}}{4}\leq\delta_1<\frac{5\delta_2-86}{14}.
\]
Hence we obtain $\delta_2>230$, so we are done. When the second condition in~\eqref{eq:condition} is satisfied, since $(\delta_2-16-6\sqrt{\delta_2-5})/4 \geq 1$, we obtain $\delta_2\geq\lceil 38+12\sqrt{6}\rceil=68$. Hence $(2+2\delta_1+\delta_2)/24 \geq (2+2+68)/24 = 3$, as required.
\end{proof}

\noindent
By the proof of Proposition~\ref{dim4_real_volume}, $P$ is a $4$-dimensional real reflexive polytope with $\vol{P}=3$ if and only if its $\delta$-vector equals $(1,1,68,1,1)$. However, by the method described in~\S\ref{subsec:find_P_from_delta} below, we can show that no such reflexive polytope exists.

The $4$-dimensional reflexive polytopes were classified by Kreuzer--Skarke~\cite{KS00}: there are $473\,800\,776$ cases. Corollary~\ref{cor:dim_4_conditions_real} makes extracting the real reflexive polytopes a simple matter, and we can recover their $\delta$-vectors. We find that the region to the left of the parabola (and closest to the $\delta_2$-axis) in Figure~\ref{fig:3d_parabola_and_tangent} is empty: all the $\delta$-vectors lie in the narrow region between the parabola and the tangent. A plot of all of the $\delta$-vectors suggests very strongly that there is an additional inequality awaiting discovery. Furthermore:

\begin{prop}\label{prop:4real_roots}
Let $P$ be a $4$-dimensional real reflexive polytope. Then the roots of $L_P$ satisfy~\eqref{eq:CS}.
\end{prop}

\noindent
As a consequence, Proposition~\ref{prop:real_reflexive_CS_volume} tells us that $\vol{P}\geq 2^4$. Unfortunately we do not have a theoretical explanation for Proposition~\ref{prop:4real_roots}.

\begin{example}\label{eg:no_connection}
There is no obvious relationship between the $\delta$-vector of $P$ and that of the dual polytope~$\dual{P}$. For example, $P=\sconv{(1,0,0,0),(0,1,0,0),(0,0,1,0),(0,0,0,1),(-2,-2,-3,-4)}$, the polytope associated with weighted projective space $X=\Proj(1,2,2,3,4)$, has $4!\vol{P}=1+2+2+3+4=12$ and $\abs{P\cap\Z^4}=7$ (the point $(-1,-1,-1,-2)$ lies on an edge of $P$) and so is a CL-polytope by Corollary~\ref{cor:dim_4_conditions}. Its dual polytope $\dual{P}$ has
\begin{align*}
4!\vol{\dual{P}}&=(-K_X)^4=\frac{12^4}{1\cdot 2^2\cdot 3\cdot 4}=432\\
\abs{\dual{P}\cap\Z^4}&=h^0(X,-K_X)\\
&=\text{coefficient of $t^{12}$ in the Taylor expansion of }\frac{1}{(1-t)(1-t^2)^2(1-t^3)(1-t^4)}\\
&=89,
\end{align*}
and hence is neither real nor a CL-polytope.
\end{example}

\begin{figure}[htbp]
\centering
\includegraphics[scale=0.8]{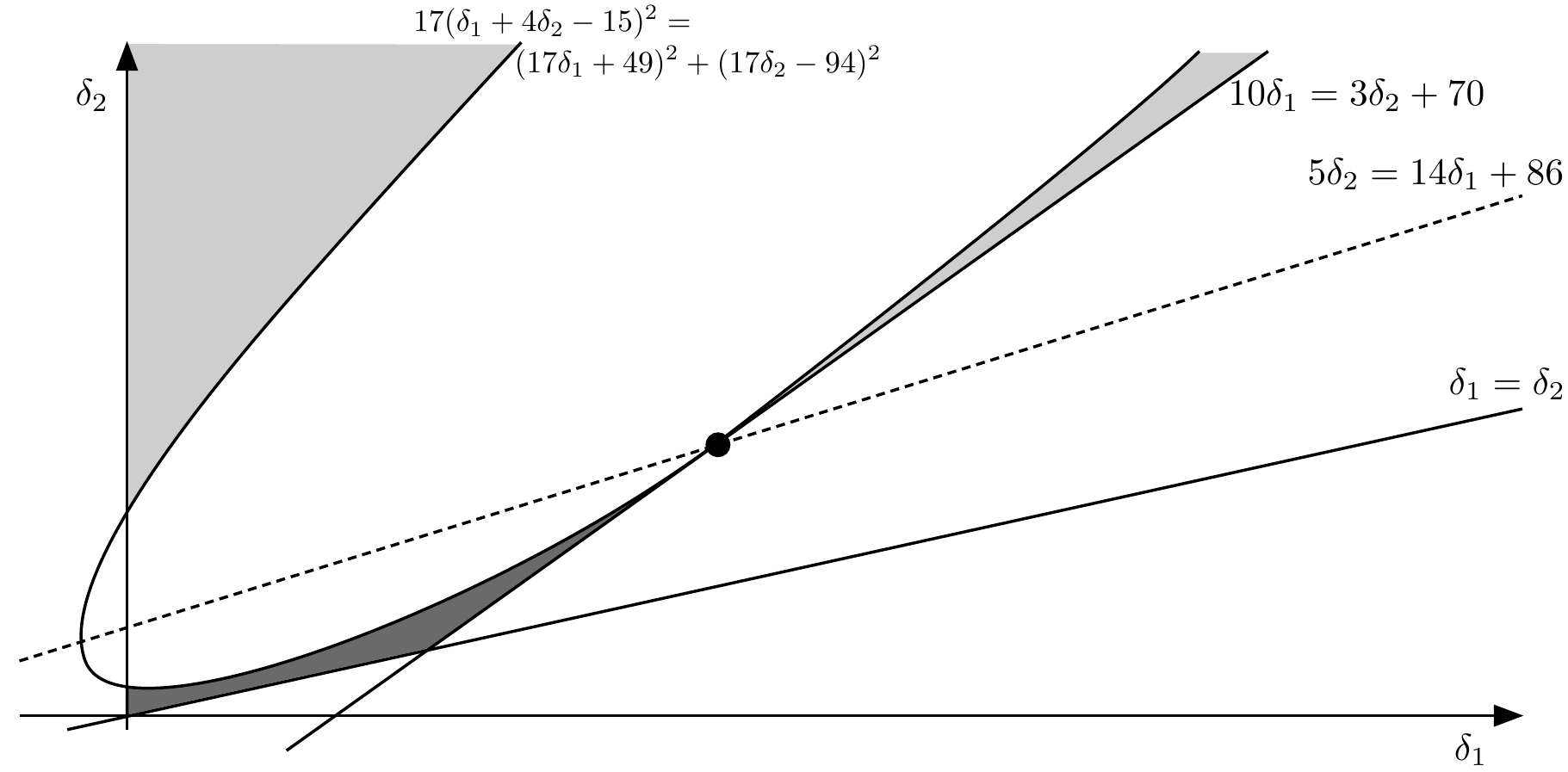}
\caption{An illustration of the regions containing the pairs $(\delta_1,\delta_2)$ for the four-dimensional CL-polytopes (the darker, bounded region closest to the origin) and the real reflexive polytopes (the lighter, unbounded regions on either side of the parabola). The dot near the centre represents the value $(76,230)$, corresponding to the Ehrhart $\delta$-vector of the $4$-cube, and gives polytopes which are simultaneously both CL-polytopes and real. The line $10\delta_1=3\delta_2+70$ is tangent to the parabola at this point.}\label{fig:3d_parabola_and_tangent}
\end{figure}

\subsection{Remaining cases}
\begin{prop}\label{dim4_mixed_volume}
Let $P$ be a $4$-dimensional reflexive polytope, and suppose that there are two roots of $L_P$ which are real, and that there are two roots with real part $-1/2$ (i.e.~we are in case~\eqref{item:dim4_case4} above). Then $\vol{P}\geq 4/3$.
\end{prop}

\begin{proof}
By considering the left-hand side of~\eqref{eq:dim4_alpha} we see that we are in case~\eqref{item:dim4_case4} if and only if $70-10\delta_1+3\delta_2\leq 0$. Since $\delta_1\leq \delta_2$, we have that $7(10-\delta_1)\leq 70-10\delta_1+3\delta_2\leq 0$. Thus $\delta_1\geq 10$. Therefore $\vol{P}=(2+2\delta_1+\delta_2)/24\geq (2+3\delta_1)/24\geq 4/3$, as required.
\end{proof}

\begin{prop}\label{prop:canonical_strip_4}
Let $P$ be a $4$-dimensional reflexive polytope, and suppose that there exists a root $-1/2+a+bi\in\C$ of $L_P$, where $a\neq 0$ and $b\neq 0$ (i.e.\ we are in case~\eqref{item:dim4_case1} above). Then $\abs{a}< 3/2$ and the roots of $L_P$ satisfy~\eqref{eq:HS}.
\end{prop}
\begin{proof}
We consider $\alpha=a+bi\in\C$ satisfying~\eqref{eq:dim4_alpha}. We need to show that $\alpha$ satisfies $-3/2 < \Re\alpha < 3/2$. Let
\[
F(z)=(2+2\delta_1+\delta_2)z^2 + \left(43+7\delta_1-\frac{5}{2}\delta_2 \right)z + \frac{105}{8}-\frac{15}{8}\delta_1+\frac{9}{16}\delta_2
\]
and consider the roots of $F$. By our hypothesis we have that $\Disc(F)<0$. Let $\beta$ and $\gamma$ be the roots of $F$, and write $\beta=re^{\theta i}$ with $r>0$ and $0<\theta<\pi$, so that $\gamma=\overline{\beta}=re^{-\theta i}$. Thus the roots $\alpha$ are given by
\[
\sqrt{r}e^{\frac{\theta}{2} i},\qquad
\sqrt{r}e^{(\pi-\frac{\theta}{2}) i},\qquad
\sqrt{r}e^{-\frac{\theta}{2} i},\qquad\text{ and }\qquad
\sqrt{r}e^{-(\pi-\frac{\theta}{2}) i}.
\]
Hence it is sufficient to show that
\[
0<\Re{\sqrt{r}e^{\frac{\theta}{2}i}}=\sqrt{r}\cos\frac{\theta}{2}=\sqrt{\frac{r+r\cos\theta}{2}}\leq\frac{3}{2}.
\]
Since $F(z)=(2+2\delta_1+\delta_2)(z-\beta)(z-\gamma)$, we have
\[
r=\sqrt{\beta \gamma}=\sqrt{\frac{\frac{105}{8}-\frac{15}{8}\delta_1+\frac{9}{16}\delta_2}{2+2\delta_1+\delta_2}}=\frac{1}{4}\sqrt{\frac{210-30\delta_1+9\delta_2}{2+2\delta_1+\delta_2}}.
\]
Moreover, one has
\[
\Disc(F) = 4(16\delta_1^2-8(\delta_2-16)\delta_1+\delta_2^2-68\delta_2+436).
\]
Let $h(\delta_1):=\Disc(F)/4$ be regarded as a quadratic in $\delta_1$. The range of $\delta_1$ such that $h(\delta_1)<0$ is given by:
\[
\frac{\delta_2-16-6\sqrt{\delta_2-5}}{4}<\delta_1<\frac{\delta_2-16+6\sqrt{\delta_2-5}}{4}.
\]
Since $\delta_1\geq 1$, we conclude that $(\delta_2-16+6 \sqrt{\delta_2-5})/4>1$. Hence $\delta_2>38-12\sqrt{6}$, and the condition that $\Disc(F)<0$ is equivalent to:
\begin{align}
\nonumber
&\phantom{m}\delta_2>38-12\sqrt{6},\qquad\text{ and}\\
\label{condi1}
&\begin{cases}
1\leq \delta_1<\dfrac{\delta_2-16+6\sqrt{\delta_2-5}}{4},
&\text{ when }\delta_2\leq 38+12\sqrt{6};\vspace{0.5em}\\
\dfrac{\delta_2-16-6\sqrt{\delta_2-5}}{4}<\delta_1<\dfrac{\delta_2-16+6\sqrt{\delta_2-5}}{4},
&\text{ when }\delta_2>38+12\sqrt{6}.
\end{cases}
\end{align}

\noindent
When $\delta_1$ and $\delta_2$ satisfy the first condition of~\eqref{condi1}, we have:
\[
\frac{210-30\delta_1+9\delta_2}{2+2\delta_1+\delta_2}=-15+\frac{24(\delta_2+10)}{2+2\delta_1+\delta_2}\leq -15+\frac{24(\delta_2+10)}{\delta_2+4}<-15+\frac{24(38-12\sqrt{6}+10)}{38-12\sqrt{6}+4}<81.
\]
When the second condition of~\eqref{condi1} is satisfied, we have:
\begin{align*}
\frac{210-30\delta_1+9\delta_2}{2+2\delta_1+\delta_2}&= -15+\frac{24(\delta_2+10)}{2+2\delta_1+\delta_2}\\
&<-15+\frac{24(\delta_2+10)}{\frac{\delta_2-16-6\sqrt{\delta_2-5}}{2}+\delta_2+2}\\
&=-15+\frac{16(\delta_2+10)}{\delta_2-4-2\sqrt{\delta_2-5}}\qquad(=:H(\delta_2))\\
&<-15+\frac{16(38+12\sqrt{6}+10)}{38+12\sqrt{6}-4-2\sqrt{38+12\sqrt{6}-5}}\\
&\qquad\qquad\text{(since }\frac{dH(\delta_2)}{d\delta_2}<0\text{ when }\delta_2>38+12\sqrt{6}\text{)}\\
&< 81
\end{align*}
We conclude that
\[
\sqrt{\frac{r+r \cos\theta}{2}}\leq\sqrt{r}=\frac{1}{2}\left(\frac{210-30\delta_1+9\delta_2}{2+2\delta_1+\delta_2}\right)^{\frac{1}{4}}< \frac{3}{2},
\]
as required.
\end{proof}

\section{Dimension five}\label{sec:dim5}
Let $P$ be a $5$-dimensional reflexive polytope with $\delta$-vector $(1,\delta_1,\delta_2,\delta_2,\delta_1,1)$. Then one of the roots of $L_P$ is $-1/2$. The remaining roots of $L_P$ fall into the four possible cases~\eqref{item:dim4_case2}--\eqref{item:dim4_case1} described at the beginning of~\S\ref{sec:dim4}.

\subsection{Five-dimensional CL-polytopes}
\begin{thm}\label{thm:dim_5_conditions}
Let $P$ be a $5$-dimensional reflexive polytope with $\delta$-vector $(1,\delta_1,\delta_2,\delta_2,\delta_1,1)$. Then $P$ is a CL-polytope if and only if either
\begin{enumerate}
\item\label{item:dim_5_conditions_1}
$\delta_1=237$ and $\delta_2=1682$, or
\item\label{item:dim_5_conditions_2}
$\delta_2<7\delta_1+23$, $71\delta_1\leq 9\delta_2+1689$, and $41(\delta_1+9\delta_2-9)^2\leq 2(41\delta_1+96)^2+2(41\delta_2-85)^2$.
\end{enumerate}
In particular, if $P$ is a CL-polytope then $\delta_1\leq 237$ and $\delta_2\leq 1682$.
\end{thm}

\begin{proof}
From Lemma~\ref{lem:main_lemma} we have
\[
L_P(z)={z+5 \choose 5} + {z \choose 5}+\delta_1\left({z+4 \choose 5} + {z+1 \choose 5}\right)+
\delta_2\left({z+3 \choose 5} + {z+2 \choose 5}\right).
\]
Since $-1/2$ is a root, we know that $L_P$ is divisible by $z+1/2$. Set $f(z):=5!L_P(z)/(z+1/2)$.
Substituting $z=-1/2+\beta i$ in $f$ and multiplying through by $5!$ gives
\[
G(\beta):=2(1+\delta_1+\delta_2)\beta^4-5(23+7\delta_1-\delta_2)\beta^2+\frac{1}{8}(1689-71\delta_1+9\delta_2).
\]
This is a quadratic in $\beta^2$, and $G(\beta)=0$ if and only if $L_P(-1/2+\beta i)=0$. The two cases follow from Lemma~\ref{lem:quadratic}. By combining the first two inequalities in case~\eqref{item:dim_5_conditions_2} we see that $\delta_1<237$ and $\delta_2<1682$; including case~\eqref{item:dim_5_conditions_1} we have that $\delta_1\leq 237$ and $\delta_2\leq 1682$.
\end{proof}

\begin{example}
The case of Theorem~\ref{thm:dim_5_conditions}\eqref{item:dim_5_conditions_1} can occur as the $\delta$-vector of the $5$-dimensional cube $\{-1,1\}^5$.
\end{example}

Once more we see that the inequalities in Theorem~\ref{thm:dim_5_conditions}\eqref{item:dim_5_conditions_2} are determined by the discriminant $\Delta$ of $G$: the parabola $41(\delta_1+9\delta_2-9)^2=2(41\delta_1+96)^2+2(41\delta_2-85)^2$. The tangent at the point $(\delta_1,\delta_2)=(237,1682)$ is given by $71\delta_1=9\delta_2+1689$, and together these two equations (or, more accurately, the corresponding two inequalities) cut out three regions in the positive quadrant. The inequality $\delta_2<7\delta_1+23$ distinguishes which of these regions contains the $\delta$-vectors for the CL-polygons and, as we shall see in Theorem~\ref{thm:dim_5_real} below, which contains the real reflexive polygons. The situation is essentially the same as the case in $4$-dimensions illustrated in Figure~\ref{fig:3d_parabola_and_tangent}. As in the $4$-dimensional case, we expect there to be a ``missing inequality'' that excludes the top-left-most region from consideration.

\begin{cor}\label{cor:dim_5_conditions}
Let $P$ be a $5$-dimensional reflexive polytope. Then $P$ is a CL-polytope if and only if either
\begin{enumerate}
\item
$\vol{P}=32$ and $\abs{P\cap\Z^5}=243$, or
\item
$15\vol{P}<2\left(\abs{P\cap\Z^5}-3\right)$, $4\abs{P\cap\Z^5}\leq 27(\vol{P}+4)$, and $\left(\abs{P\cap\Z^5}-3-6 \vol{P}\right)^2\geq 72\vol{P}$.
\end{enumerate}
\end{cor}

\begin{proof}
Simply make the substitutions $2(1+\delta_1+\delta_2)=5!\vol{P}$ and $\delta_1=\abs{P\cap\Z^5}-6$ in Theorem~\ref{thm:dim_5_conditions}.
\end{proof}

\subsection{Five-dimensional real reflexive polytopes}
\begin{thm}\label{thm:dim_5_real}
Let $P$ be a $5$-dimensional reflexive polytope with $\delta$-vector $(1,\delta_1,\delta_2,\delta_2,\delta_1,1)$. Then $P$ is real if and only if either
\begin{enumerate}
\item\label{item:dim_5_real_1}
$\delta_1=237$ and $\delta_2=1632$, or
\item\label{item:dim_5_real_2}
$\delta_2>7\delta_1+23$, $71\delta_1\leq 9\delta_2+1689$, and $41(\delta_1+9\delta_2-9)^2\leq 2(41\delta_1+96)^2+2(41\delta_2-85)^2$.
\end{enumerate}
\end{thm}

\begin{proof}
Substituting $z=-1/2+\alpha$, where $\alpha\in\R$, into $f(z)$ in the proof of Theorem~\ref{thm:dim_5_conditions} gives
\begin{equation}\label{eq:dim5_alpha}
2(1+\delta_1+\delta_2)\alpha^4+5(23+7\delta_1-\delta_2)\alpha^2+\frac{1}{8}(1689-71\delta_1+9\delta_2)=0.
\end{equation}
This is a quadratic in $\alpha^2$ and the result follows from Lemma~\ref{lem:quadratic}.
\end{proof}

\begin{cor}\label{cor:dim_5_real}
Let $P$ be a $5$-dimensional reflexive polytope. Then $P$ is a CL-polytope if and only if either
\begin{enumerate}
\item
$\vol{P}=32$ and $\abs{P\cap\Z^5}=243$, or
\item
$15\vol{P}>2\left(\abs{P\cap\Z^5}-3\right)$, $4\abs{P\cap\Z^5}\leq 27(\vol{P}+4)$, and $\left(\abs{P\cap\Z^5}-3-6\vol{P}\right)^2\geq 72\vol{P}$.
\end{enumerate}
\end{cor}

\begin{prop}\label{dim5_real_volume}
Let $P$ be a $5$-dimensional real reflexive polytope. Then $\vol{P}\geq 16/5$.
\end{prop}

\begin{proof}
By Theorem~\ref{thm:dim_5_real}, $P$ is real if and only if $\delta_1$ and $\delta_2$ satisfy either condition~\eqref{item:dim_5_real_1}, or condition~\eqref{item:dim_5_real_2}. In case~\eqref{item:dim_5_real_1}, since $\vol{P}=2(1+\delta_1+\delta_2)/120$, we are done. In case~\eqref{item:dim_5_real_2}, by the third inequality $41(\delta_1+9\delta_2-9)^2\leq 2(41\delta_1+96)^2+2(41\delta_2-85)^2$, we have
\begin{align}\label{eq:condition_5}
\delta_1\geq\frac{3\delta_2-67+20\sqrt{3\delta_2-5}}{27},\qquad\text{ or }\qquad
1\leq\delta_1\leq\frac{3\delta_2-67-20\sqrt{3\delta_2-5}}{27}.
\end{align}
Moreover, by inequality $\delta_2>7\delta_1+23$, we have that $\delta_1<(\delta_2-23)/7$.

When the first condition in~\eqref{eq:condition_5} is satisfied, using $\delta_1\leq\delta_2$ we see that
\[
\frac{3\delta_2-67+20\sqrt{3\delta_2-5}}{27}\leq\delta_1<\frac{\delta_2-23}{7}.
\]
Hence we obtain $\delta_2>1682$, and so we are done. When the second condition in~\eqref{eq:condition_5} is satisfied, using the fact that $(3\delta_2-67-20\sqrt{3\delta_2-5})/27\geq 1$,
we obtain $\delta_2\geq\lceil 98+20\sqrt{21}\rceil=190$. Hence $2(1+\delta_1+\delta_2)/120\geq (1+1+190)/60=16/5$, as required.
\end{proof}

\noindent
By the proof of Proposition~\ref{dim5_real_volume}, $P$ is a $5$-dimensional real reflexive polytope with $\vol{P}=16/5$ if and only if its $\delta$-vector equals $(1,1,190,190,1,1)$. One can show that no such reflexive polytopes exist by the method of~\S\ref{subsec:find_P_from_delta} below.

\subsection{Remaining cases}
\begin{prop}\label{dim5_mixed_volume}
Let $P$ be a $5$-dimensional polytope, and suppose that there are two roots of $L_P$ which are real, and that there are two roots with real part $-1/2$. Then $\vol{P}\geq 19/20$.
\end{prop}

\begin{proof}
By considering the left-hand side of~\eqref{eq:dim5_alpha}, we see that there are two real roots and two roots with real part $-1/2$ if and only if $1689-71\delta_1+9\delta_2\leq 0$. Since $\delta_1\leq\delta_2$, we have that $0\geq 1689-71\delta_1+9\delta_2\geq 1689-62\delta_1$. Thus $\delta_1\geq\lceil 1689/62\rceil = 28$. Therefore $\vol{P}=2(1+\delta_1+\delta_2)/120\geq (1+2\delta_1)/60\geq 19/20$.
\end{proof}

\begin{prop}\label{prop:canonical_strip_5}
Let $P$ be a $5$-dimensional reflexive polytope, and suppose that there exists a root $-1/2+a+bi\in\C$ of $L_P$, where $a\neq 0$ and $b\neq 0$. Then $\abs{a}<3/2$ and the roots of $L_P$ satisfy~\eqref{eq:HS}.
\end{prop}

\begin{proof}
We consider $\alpha$ satisfying equation~\eqref{eq:dim5_alpha}; we need to prove that $\alpha$ satisfies $-3/2<\Re{\alpha}<3/2$. Set
\[
F(z):=2(1+\delta_1+\delta_2)z^2+5(23+7\delta_1-\delta_2)z+\frac{1689}{8}-\frac{71}{8}\delta_1+\frac{9}{8}\delta_2,
\]
and consider the roots of $F$. By the hypothesis we have that $\Disc(F)<0$. We shall show that
\[
\sqrt{\frac{r+r \cos \theta}{2}}<\frac{3}{2},
\]
where $\beta=re^{\theta i}$ is a root of $F$, $r>0$, $0<\theta<\pi$. Notice that
\[
r=\sqrt{\beta \overline{\beta}}=\frac{1}{4}\sqrt{\frac{1689-71\delta_1+9\delta_2}{1+\delta_1+\delta_2}},
\qquad\text{ and }\qquad
r\cos\theta=\frac{\beta+\overline{\beta}}{2}=-\frac{1}{4}\cdot\frac{5(23+7\delta_1-\delta_2)}{1+\delta_1+\delta_2}.
\]
Moreover, one has that
\[
\Disc(F)=\frac{16}{41}\left(2(41\delta_1+96)^2+2(41\delta_2-85)^2-41(\delta_1+9\delta_2-9)^2\right).
\]
Set $h(\delta_1):=\Disc(F)/16$ to be a polynomial in $\delta_1$. We require that $h(\delta_1)<0$. The range of values for $\delta_1$ such that $h(\delta_1)<0$ is:
\[
\frac{3\delta_2-67-20\sqrt{3\delta_2-5}}{27}<\delta_1<\frac{3\delta_2-67+20\sqrt{3\delta_2-5}}{27}.
\]
Since $\delta_1\geq 1$, we have that $(3\delta_2-67+20\sqrt{3\delta_2-5})/27>1$. Thus $\delta_2>98-20\sqrt{21}$. Hence the condition $\Disc(F)<0$ is equivalent to:
\begin{align}
\nonumber
&\phantom{m}\delta_2>98-20\sqrt{21},\qquad\text{ and}\\
\label{condi2}
&\begin{cases}
1\leq\delta_1<\dfrac{3\delta_2-67+20\sqrt{3\delta_2-5}}{27},
&\text{ when }\delta_2\leq 98+20\sqrt{21};\vspace{0.5em}\\
\dfrac{3\delta_2-67-20\sqrt{3\delta_2-5}}{27}<\delta_1<\dfrac{3\delta_2-67+20\sqrt{3\delta_2-5}}{27},
&\text{ when }\delta_2>98+20\sqrt{21}.
\end{cases}
\end{align}

\noindent
When the first condition of~\eqref{condi2} is satisfied, we have:
\begin{align*}
\sqrt{\frac{1689-71\delta_1+9\delta_2}{1+\delta_1+\delta_2}}-\frac{5(23+7\delta_1-\delta_2)}{1+\delta_1+\delta_2}
&=\sqrt{-71+\frac{80(\delta_2+22)}{1+\delta_1+\delta_2}}+\frac{40(\delta_2-2)}{1+\delta_1+\delta_2}-35\\
&\leq\sqrt{-71+\frac{80(\delta_2+22)}{\delta_2+2}}+\frac{40(\delta_2-2)}{\delta_2+2}-35\qquad(=:H_1(\delta_2))\\
&\leq\sqrt{-71+\frac{80(98+22)}{98+2}}+\frac{40(98-2)}{98+2}-35\\
&\qquad\qquad\text{(since }\frac{dH_1(\delta_2)}{d\delta_2}>0\text{ when }\delta_2<98\\
&\qquad\qquad\qquad\text{ and }\frac{dH_1(\delta_2)}{d\delta_2}<0\text{ when }\delta_2>98\text{)}\\
&<18
\end{align*}
We obtain:
\[
\sqrt{\frac{r+r \cos \theta}{2}}<\sqrt{\frac{18}{4}\cdot\frac{1}{2}}=\frac{3}{2}.
\]

\noindent
When the second condition of~\eqref{condi2} is satisfied, we have:
\begin{align*}
\frac{1689-71\delta_1+9\delta_2}{1+\delta_1+\delta_2}
&=-71+\frac{80(\delta_2+22)}{1+\delta_1+\delta_2}\\
&<-71+\frac{80(\delta_2+22)}{\frac{3\delta_2-67-20 \sqrt{3\delta_2-5}}{27}+\delta_2+2}\\
&=-71+\frac{2160(\delta_2+22)}{30\delta_2-13-20\sqrt{3\delta_2-5}}\qquad(=:H_2(\delta_2))\\
&<-71+\frac{2160(98+20\sqrt{21}+22)}{30(98+20\sqrt{21})-13-20\sqrt{3(98+20\sqrt{21})-5}}\\
&\qquad\qquad\text{(since }\frac{dH_2(\delta_2)}{d\delta_2}<0\text{ when }\delta_2\geq 98+20\sqrt{21}\text{)}\\
&< 81
\end{align*}
We conclude that
\[
\sqrt{\frac{r+r\cos\theta}{2}}\leq\sqrt{r}=\frac{1}{2}\left(\frac{1689-71\delta_1+9\delta_2}{1+\delta_1+\delta_2}\right)^{\frac{1}{4}}<\frac{3}{2},
\]
as required.
\end{proof}

\section{Dimensions six and seven}\label{sec:dim67}
If $P$ is a $(2k+1)$-dimensional reflexive polytope then $-1/2$ is a root of $L_P$. This we shall consider both $6$- and $7$-dimensional reflexive polytopes together. We shall require the following lemma; this is simply an application of Descartes' rule of signs.

\begin{lemma}\label{lem:cubic}
Let $A,B,C,D\in\R$ be real numbers, $A>0$. All solutions of the equation
\[
Az^3+Bz^2+Cz+D=0
\]
are positive real numbers if and only if $B<0$, $C>0$, $D\leq 0$, and the discriminant $\Disc\geq 0$.
\end{lemma}

\begin{lemma}\label{lem:cubic2}
Let $A,B,C,D\in\R$ be real numbers, $A>0$. All solutions of the equation
\begin{equation}\label{eq:dim6pol2}
Az^3+Bz^2+Cz+D=0
\end{equation}
are non-negative real numbers if and only if one of the following conditions is satisfied:
\begin{enumerate}
\item\label{equal}
$B=C=D=0$;
\item\label{equal2}
$B<0$ and $C=D=0$; or
\item\label{equal3}
$\Disc\geq 0$, $B<0$, $C>0$, and $D<0$.
\end{enumerate}
\end{lemma}

\begin{proof}
First suppose that all roots of~\eqref{eq:dim6pol2} are non-negative real numbers. Then it is easy to see that one of the conditions~\eqref{equal},~\eqref{equal2}, or~\eqref{equal3} is satisfied. Conversely, suppose that one of~\eqref{equal},~\eqref{equal2}, or~\eqref{equal3} holds. Let $B'>0$, $C'<0$, and $D'>0$ denote a small perturbation of, respectively, the coefficients  $B\geq 0$, $C\leq 0$, and $D\geq 0$. Then it follows from Rouche's Theorem that the roots of 
\begin{align}\label{eq:dim6pol3}
A\gamma^3+B'\gamma^2+C'\gamma+D'=0
\end{align}
are very close to the roots of~\eqref{eq:dim6pol2}, and Lemma \ref{lem:cubic} implies that all roots of~\eqref{eq:dim6pol3} are positive. This perturbation method give us a sequence of roots of a polynomial sequence, and the limit of this sequence of roots is the roots of~\eqref{eq:dim6pol2}.
\end{proof}

In~\S\S\ref{sec:dim6_7_CL}--\ref{sec:dim6_7_real} we will make use of the following two polynomials:
\begin{align*}
f_6(z)&:=(2+2\delta_1+2\delta_2+\delta_3)z^3-\frac{5}{4}(202+82\delta_1+10\delta_2-7\delta_3)z^2\\
&\hspace{3em}+\frac{1}{16}(24278+1478\delta_1-682\delta_2+259\delta_3)z-\frac{45}{64}(462-42\delta_1+14\delta_2-5\delta_3)\\
f_7(z)&:=(1+\delta_1+\delta_2+\delta_3)z^3-\frac{7}{4}(139+67\delta_1+19\delta_2-5\delta_3)z^2\\
&\hspace{3em}+\frac{7}{16}(8197+1237\delta_1-203\delta_2+37\delta_3)z-\frac{3}{64}(88069-3043\delta_1+429\delta_2-75\delta_3)
\end{align*}

\subsection{Six- and seven-dimensional CL-polytopes}\label{sec:dim6_7_CL}
\begin{thm}\label{thm:dim6_conditions}
Let $P$ be a $6$-dimensional reflexive polytope with $\delta$-vector $(1,\delta_1,\delta_2,\delta_3,\delta_2,\delta_1,1)$. Then $P$ is a  CL-polytope if and only if one of the following holds:
\begin{enumerate}
\item\label{thm:dim6_conditions_1}
$\delta_1=722$, $\delta_2=10543$, and $\delta_3=23548$;
\item\label{thm:dim6_conditions_2} 
$994\delta_1=27\delta_3+81872$, $497\delta_2=218\delta_3+106407$, and $\delta_3<23548$; or
\item\label{thm:dim6_conditions_3}
$202+82\delta_1+10\delta_2>7\delta_3$, $24278+1478\delta_1+259\delta_3>682\delta_2$, $462+14\delta_2>42\delta_1+5\delta_3$, and $\Disc\left(f_6\right)\geq 0$.
\end{enumerate}
In particular, if $P$ is a CL-polytope then $\delta_1\leq 722$, $\delta_2\leq 10543$, and $\delta_3\leq 23548$.
\end{thm}

\begin{proof}
From Lemma~\ref{lem:main_lemma} we obtain
\begin{align}\label{eq:dim6}
L_P(z)={z+6 \choose 6} + {z \choose 6}+\delta_1\left({z+5 \choose 6} + {z+1 \choose 6} \right)+
\delta_2\left({z+4 \choose 6} + {z+2 \choose 6} \right)+\delta_3{z+3 \choose 6}.
\end{align}
Substituting $z=-1/2+\beta i$ in $L_P(z)$ and multiplying by $-6!$ gives:
\begin{align*}
G(\beta):=(2+2\delta_1+2\delta_2+&\delta_3)\beta^6-\frac{5}{4}(202+82\delta_1+10\delta_2-7\delta_3)\beta^4\\
&+\frac{1}{16}(24278+1478\delta_1-682\delta_2+259\delta_3)\beta^2-\frac{45}{64}(462-42\delta_1+14\delta_2-5\delta_3).
\end{align*}
This is a cubic in $\beta^2$, with $G(\beta)=f_6(\beta^2)$, and $G(\beta)=0$ if and only if $L_P(-1/2+\beta i) = 0$. The three cases follow from applying Lemma~\ref{lem:cubic2} to $f_6$. The bounds $\delta_1\leq 722$, $\delta_2\leq 10543$, and $\delta_3\leq 23548$ are obvious in cases~\eqref{thm:dim6_conditions_1} and~\eqref{thm:dim6_conditions_2}; in case~\eqref{thm:dim6_conditions_3} it follows through a positive linear combination of the first three inequalities.
\end{proof}

\begin{thm}\label{thm:dim7_conditions}
Let $P$ be a $7$-dimensional reflexive polytope with $\delta$-vector $(1,\delta_1,\delta_2,\delta_3,\delta_3,\delta_2,\delta_1,1)$. Then $P$ is a CL-polytope if and only if either
\begin{enumerate}
\item
$\delta_1=2179$, $\delta_2=60657$, and $\delta_3=259723$;
\item
$10882\delta_1=81\delta_3+2674315$,
$10882\delta_2=2477\delta_3+16735603$, and
$\delta_3<259723$; or
\item\label{thm:dim7_conditions_3}
$139+67\delta_1+19\delta_2>5\delta_3$, $8197+1237\delta_1+37\delta_3>203\delta_2$, $88069+429\delta_2>3043\delta_1+75\delta_3$, and $\Disc\left(f_7\right)\geq 0$.
\end{enumerate}
In particular, if $P$ is a CL-polytope then $\delta_1\leq 2179$, $\delta_2\leq 60657$, and $\delta_3\leq 259723$.
\end{thm}

\begin{proof}
From Lemma~\ref{lem:main_lemma}, we obtain
\begin{align*}
L_P(z)={z+7\choose 7}+&{z\choose 7}+\delta_1\left({z+6\choose 7}+{z+1\choose 7}\right)+\\
&\delta_2\left({z+5\choose 7} + {z+2\choose 7}\right)+\delta_3\left({z+4\choose 7}+{z+3\choose 7} \right).
\end{align*}
Then $L_P(z)$ can be divided by $z+1/2$ (because $-1/2$ is a root). Set $f(z):=-7!L_P(z)/2(z+1/2)$. The result also follows from Lemma~\ref{lem:cubic2} by considering $f(-1/2+\beta i)=f_7(\beta^2)$. Again, in case~\eqref{thm:dim7_conditions_3} the bounds $\delta_1\leq 2179$, $\delta_2\leq 60657$, and $\delta_3\leq 259723$ can be obtained from a positive linear combination of the first three inequalities.
\end{proof}

\begin{cor}
Let $P$ be a $6$- or $7$-dimensional CL-polytope. Then, respectively, $\abs{P \cap \Z^6} \leq 3^6$ or $\abs{P \cap \Z^7} \leq 3^7$.
\end{cor}

\subsection{Six- and seven-dimensional real reflexive polytopes}\label{sec:dim6_7_real}
By making the substitution $z=-1/2+\alpha$ in $L_P(z)$, we readily obtain the following two theorems:

\begin{thm}\label{thm:dim6_real_conditions}
Let $P$ be a $6$-dimensional reflexive polytope with $\delta$-vector $(1,\delta_1,\delta_2,\delta_3,\delta_2,\delta_1,1)$. Then $P$ is real if and only if one of the following holds:
\begin{enumerate}
\item
$\delta_1=722$, $\delta_2=10543$, and $\delta_3=23548$;
\item 
$994\delta_1=27\delta_3+81872$, $497\delta_2=218\delta_3+106407$, and $\delta_3>23548$; or
\item
$202+82\delta_1+10\delta_2>7\delta_3$, $24278+1478\delta_1+259\delta_3>682\delta_2$, $462+14\delta_2<42\delta_1+5\delta_3$, and $\Disc\left(f_6\right)\geq 0$.
\end{enumerate}
\end{thm}

\begin{proof}
The proof is similar to Theorem~\ref{thm:dim6_conditions} above; making the substitution $z=-1/2+\alpha$, $\alpha\in\R$, in~\eqref{eq:dim6} and multiplying through by $6!$ we obtain a cubic in $\alpha^2$ given by $g_6(\alpha^2)$, where
\begin{align*}
g_6(z):=(2+2\delta_1+2\delta_2+&\delta_3)z^3+\frac{5}{4}(202+82\delta_1+10\delta_2-7\delta_3)z^2\\
&+\frac{1}{16}(24278+1478\delta_1-682\delta_2+259\delta_3)z+\frac{45}{64}(462-42\delta_1+14\delta_2-5\delta_3).
\end{align*}
The result follows from Lemma~\ref{lem:cubic2}, and noticing that $\Disc(f_6)=\Disc(g_6)$.
\end{proof}

\begin{thm}\label{thm:dim7_real_conditions}
Let $P$ be a $7$-dimensional reflexive polytope with $\delta$-vector $(1,\delta_1,\delta_2,\delta_3,\delta_3,\delta_2,\delta_1,1)$. Then $P$ is real if and only if either
\begin{enumerate}
\item
$\delta_1=2179$, $\delta_2=60657$, and $\delta_3=259723$;
\item
$10882\delta_1=81\delta_3+2674315$,
$10882\delta_2=2477\delta_3+16735603$, and
$\delta_3>259723$; or
\item
$139+67\delta_1+19\delta_2<5\delta_3$, $8197+1237\delta_1+37\delta_3>203\delta_2$, $88069+429\delta_2<3043\delta_1+75\delta_3$, and $\Disc\left(f_7\right)\geq 0$.
\end{enumerate}
\end{thm}

\begin{proof}
Again, the proof parallels that of Theorem~\ref{thm:dim7_conditions}. In this case we obtain the cubic in $\alpha^2$ given by $g_7(\alpha^2)$, where
\begin{align*}
g_7(z):=(1+\delta_1+\delta_2+&\delta_3)z^3+\frac{7}{4}(139+67\delta_1+19\delta_2-5\delta_3)z^2\\
&+\frac{7}{16}(8197+1237\delta_1-203\delta_2+37\delta_3)z+\frac{3}{64}(88069-3043\delta_1+429\delta_2-75\delta_3).
\end{align*}
Since $\Disc(f_7)=\Disc(g_7)$, the result follows by Lemma~\ref{lem:cubic2}.
\end{proof}

\section{Conjectures in higher dimensions}
In light of the above results, the following two conjectures seem natural:

\begin{conjecture}
Let $P$ be a $d$-dimensional CL-polytope. Then $\abs{P\cap {\Z}^d}\leq 3^d$.
\end{conjecture}

\begin{conjecture}
Let $P$ be a $d$-dimensional CL-polytope with $\delta$-vector $(\delta_0,\ldots,\delta_d)$. Then
\[
\delta_i\leq \sum_{j=0}^d{d+1\choose 2i-j}A(d,j),\qquad\text{ for each }i\in\{0,\ldots,d\}.
\]
\end{conjecture}

\noindent
We have already shown that these conjectures hold when $d\leq 7$. Slightly less confidently, we suggest:

\begin{conjecture}
Let $P$ be a $d$-dimensional real reflexive polytope. Then $\vol{P}\geq 2^d$.
\end{conjecture}

\noindent
In~\S\ref{sec:dim_4_real_reflexive} we shown that this conjecture holds when $d=4$, however this is achieved by using the classification of Kreuzer--Skarke~\cite{KS00} to show that the roots of $L_P$ satisfy~\eqref{eq:CS} for every $4$-dimensional real reflexive polytope $P$, and lacks a theoretical explanation. Experimentation suggests the following:

\begin{conjecture}
Let $P$ be a $6$- or $7$-dimensional reflexive polytope, and suppose that there exists a root $-1/2+a+bi\in\C$ of $L_P$, where $a\neq 0$ and $b\neq 0$. Then $\abs{a}<5/2$.
\end{conjecture}

\noindent
If this conjecture holds then we can conclude that hypothesis~\eqref{eq:HS} holds when $d\leq 7$; however see~\S\ref{sec:canonical_strip} for an example when $d=10$ whose roots fails to satisfy~\eqref{eq:HS}.

\section{Hypothesis~\eqref{eq:HS} and an example in dimension ten}\label{sec:canonical_strip}
\subsection{Hypothesis~\eqref{eq:HS}}
In this section we prove Theorem~\ref{thm:reflexive_strip}. We actually prove a stronger result:

\begin{thm}\label{thm_h}
Let $P$ be a $d$-dimensional reflexive polytope and let $\alpha\in\C$ be a root of $L_P$.
\begin{enumerate}
\item\label{item:thm_h_1}
If $\alpha\in\R$ then $\alpha$ satisfies $-\!\left\lfloor d/2\right\rfloor<\alpha<\left\lfloor d/2\right\rfloor-1$.
\item\label{item:thm_h_2}
If $d\leq 5$ then $\alpha$ satisfies $-\!\left\lfloor d/2\right\rfloor<\Re{\alpha}<\left\lfloor d/2\right\rfloor-1$.
\end{enumerate}
\end{thm}

\begin{proof}
First we prove case~\eqref{item:thm_h_1}. Let $(1,\delta_1,\ldots,\delta_{\lfloor d/2 \rfloor},\ldots,\delta_1,1)$ be the $\delta$-vector of $P$, where $\delta_i\in\Z_{\geq 1}$ for each $1\leq i\leq\lfloor d/2\rfloor$. By Lemma~\ref{lem:main_lemma} we have that $L_P(z)=\sum_{i=0}^d\delta_i{n+d-i\choose d}$. Define
\begin{align*}
N_i(z)&:=\prod_{j=0}^{d-1}(z+d-i-j)+\prod_{j=0}^{d-1}(z+i-j),\qquad\text{ for each }0\leq i\leq\lfloor d/2\rfloor-1,
\intertext{and}
N_{\lfloor d/2 \rfloor}(z)&:=
\begin{cases}
\displaystyle\prod_{j=0}^{d-1}\left(z+ \frac{d}{2}-j\right),&\text{ when $d$ is even};\vspace{0.5em}\\
\displaystyle\prod_{j=0}^{d-1}\left(z+\frac{d+1}{2}-j\right)+\prod_{j=1}^d\left(z+\frac{d-1}{2}-j\right),&\text{ when $d$ is odd}.
\end{cases}
\end{align*}
Then $L_P(z)=\sum_{i=0}^{\lfloor d/2\rfloor}\delta_iN_i(z)/d!$. Set
\[
f(z):=d!L_P\left(z-\frac{1}{2}\right)=\sum_{i=0}^{\lfloor d/2\rfloor}\delta_iN_i\left(z-\frac{1}{2}\right).
\]
It is sufficient to prove that all the real roots of $f$ are contained in the open interval $(-\lfloor d/2\rfloor +1/2,\lfloor d/2\rfloor-1/2)$. Note that $f$ satisfies
\begin{equation}\label{funceq2}
f(z)=(-1)^df(-z).
\end{equation}

\noindent
For $N_i(z-1/2)$, $0\leq i\leq\lfloor d/2\rfloor$, we have the following equalities:
\begin{align*}
N_i\left(z-\frac{1}{2}\right)&=\prod_{j=0}^{d-1}\left(z+d-\frac{1}{2}-i-j\right)+\prod_{j=0}^{d-1}\left(z-\frac{1}{2}+i-j\right)\\
&=\prod_{l=0}^{2i-1}\left(z-\frac{1}{2}+i-l\right)\left(\prod_{j=0}^{d-2i-1}\left(z-\frac{1}{2}+d-i-j\right)+\prod_{j=0}^{d-2i-1}\left(z-\frac{1}{2}-i-j\right)\right)\\
&=\prod_{l=0}^{2i-1}\left(z-\frac{1}{2}+i-l\right)M_i(z),
\intertext{where}
M_i(z)&:=\prod_{j=0}^{d-2i-1}\left(z+\frac{1}{2}+i+j\right)+\prod_{j=0}^{d-2i-1}\left(z-\frac{1}{2}-i-j\right).
\end{align*}

\noindent
The coefficients of $M_i$ are positive rational numbers when the parity of the degree of $z$ is the same as the parity of $d$, and its coefficients are $0$ when the parity of the degree of $z$ is different with the parity of $d$. Let $\alpha$ be a real number with $\alpha\geq\lfloor d/2\rfloor-1/2$. Since $\alpha>0$ we have that $M_i(\alpha)>0$. In addition, $\prod_{l=0}^{2i-1}(\alpha-(1/2-i+l))>0$, since $0\leq l\leq 2i-1$ and $0\leq i\leq\lfloor d/2\rfloor$. Hence $\alpha$ cannot be a root of $f$ (since each $\delta_i$ is positive). Moreover, by~\eqref{funceq2}, for a real number $\beta$ with $\beta\leq-\left\lfloor d/2\right\rfloor+1/2$, $\beta$ also cannot be a root of $f(z)$.

We have already shown that case~\eqref{item:thm_h_2} holds when $d=2$ and $d=3$ (Propositions~\ref{prop:dim_2_dicotomy} and~\ref{prop:dim_3_dicotomy}, respectively). When $d=4$ the roots of $L_P(z)$ fall into the four cases~\eqref{item:dim4_case2}--\eqref{item:dim4_case1} described at the beginning of~\S\ref{sec:dim4}. Of these,~\eqref{item:dim4_case2}--\eqref{item:dim4_case4} satisfy~\eqref{item:thm_h_2}: either a root $\alpha$ is of the form $-1/2\pm bi$, and so $\Re{\alpha}=-1/2$, or $\alpha$ is real and so is covered by~\eqref{item:thm_h_1}. The only remaining possibility is~\eqref{item:dim4_case1}, in which case Proposition~\ref{prop:canonical_strip_4} gives the result. Similarly, when $d=5$ the result follows from Proposition~\ref{prop:canonical_strip_5}.
\end{proof}

\subsection{Realising a simplex from a $\delta$-vector}\label{subsec:find_P_from_delta}
In~\S\ref{subsec:dim_10_example} we give a $10$-dimensional reflexive polytope failing to satisfy~\eqref{eq:HS}. Finding an integer-valued palindromic vector whose numerics give an example is straight-forward; the difficulty lies in showing that this vector is the $\delta$-vector for a lattice polytope. The method we describe below can be used to find all lattice polytopes $P$ with palindromic vector of the form $(1,1,\delta_2,\ldots,\delta_{\lfloor d/2\rfloor},\ldots,\delta_2,1,1)$ of length $d+1$, if such $P$ exist, and otherwise to prove that there are no such $P$.

Let $\delta=(1,1,\delta_2,\ldots,\delta_{\lfloor d/2\rfloor},\ldots,\delta_2,1,1)$ be of length $d+1$, where $\delta_2,\ldots,\delta_{\lfloor d/2\rfloor}\in\Z_{\geq 1}$, and suppose that this is the $\delta$-vector for a lattice polytope $P$ of dimension $d$. Since $\delta$ is palindromic, $P$ is necessarily reflexive. In particular the origin is the only (strict) interior lattice point of $P$. Since $\delta_1=\abs{P\cap\Z^d}-d-1=1$ we know that $P$ is a simplex with $\bdry{P}\cap\Z^d=\V{P}$. Hence $P$ is a terminal reflexive Fano simplex (see~\cite{KN12} for a survey of Fano polytopes). The corresponding toric variety $X$ given by the spanning fan of $P$ is a Gorenstein fake weighted projective space: there exists some weighted projective space $Y=\Proj(\lambda_0,\ldots,\lambda_d)$ with well-formed weights $(\lambda_0,\ldots,\lambda_d)\in\Z_{\geq 0}^d$ such that $X$ is the quotient of $Y$ by the action of a finite group $G$ acting free in codimension one. For details see~\cite{Buc08,Con02,Kas08b}. By \emph{well-formed} we mean that $\gcd{\lambda_0,\ldots,\widehat{\lambda_i},\ldots,\lambda_d}=1$ for each $i\in\{0,\ldots,d\}$, where $\widehat{\lambda_i}$ indicates that the $i$-th weight $\lambda_i$ is omitted.

Let $N'\subset\Z^d$ be the sublattice generated by $\V{P}$. Then $G=\Z^d/N'$ and the order of $G$ is given by the index
\[
\mult{P}:=\left[\Z^d:N'\right],
\]
which we call the \emph{multiplicity} of $P$. If we restrict $P$ to the sublattice $N'$, we recover (up to isomorphism) the simplex $Q$ associated with the weighted projective space $Y$~\cite[Proposition~2]{BB92}. Moreover, there exists a matrix $H$ in Hermite normal form with $\det{H}=\mult{P}$ such that
\begin{equation}\label{eq:wps_hermite}
P\cong Q\cdot H.
\end{equation}
Hence $\vol{P}=\mult{P}\cdot\vol{Q}$. In particular,
\begin{equation}\label{wps:poss_mults}
\mult{P}\mid d!\vol{P}=\sum_{i=0}^d\delta_i,
\end{equation}
and so the possible multiplicities are determined by $\delta$. The volume $d!\vol{Q}=h$, where $h:=\lambda_0+\cdots+\lambda_d$ is defined to be the sum of the weights of $Q$ (and is equal to the Fano index of $Y$), so that $h$ is uniquely determined for each choice of multiplicity.

Assume that we have chosen a multiplicity, and consequently have fixed a value for $h$. Since $P$ is terminal and reflexive we have that $Q$ must be terminal and reflexive~\cite[Corollary~2.4 and Corollary~2.5]{Kas08b}. By, for example,~\cite[Lemma~3.5.6]{CK-MirrorSymmetry}, we have that $Q$ is reflexive if and only if
\begin{equation}\label{eq:wps_reflexive}
\lambda_i\mid h,\qquad\text{ for each }i\in\{0,\ldots,d\}.
\end{equation}
Without loss of generality we may assume that the weights are ordered $\lambda_0\leq\ldots\leq\lambda_d$. Since $Q$ is terminal, by~\cite[Theorem~3.5]{Kas08b} we have the strict inequalities
\begin{equation}\label{eq:wps_terminal}
\frac{\lambda_i}{h}<\frac{1}{d-i+2},\qquad\text{ for each }i\in\{2,\ldots,d\}.
\end{equation}
Thus the possible choices of weights for $Q$ are greatly restricted, and can easily be listed for any given $h$ using~\eqref{eq:wps_reflexive} and~\eqref{eq:wps_terminal}.

Suppose now that we have chosen some possible weights $(\lambda_0,\ldots,\lambda_d)$ for $Q$. These weights must satisfy three conditions, as follows.
\begin{enumerate}
\item\label{item:wps_iff_terminal}
The inequalities given by~\eqref{eq:wps_terminal} are necessary but not sufficient to guarantee that $Q$ is terminal, however this is easily verified~\cite[Proposition~3.3]{Kas08c}: $Q$ is terminal if and only if
\[
\sum_{i=0}^d\left\{\frac{\lambda_i\kappa}{h}\right\}\in\{2,\ldots,d-1\},\qquad\text{ for each }\kappa\in\{2,\ldots,h-2\},
\]
where $\{a/b\}$ denotes the fractional part of $a/b\in\Q$.
\item\label{item:wps_divides_degree}
Since $P$ is reflexive, we have that $d!\vol{\dual{P}}\in\Z$. But $d!\vol{\dual{P}}=d!\vol{\dual{Q}}/\mult{P}$, and $d!\vol{\dual{Q}}$ is simply the anticanonical degree of $Y$, given by
\[
d!\vol{\dual{Q}}=(-K_Y)^d=\frac{h^d}{\prod_{i=0}^d\lambda_i}.
\]
Hence we have the requirement that:
\[
\mult{P}\,\Big|\,\frac{h^d}{\prod_{i=0}^d\lambda_i}.
\]
\item\label{item:wps_ehrhart_bounds}
Since $N'\hookrightarrow\Z^d$ we require that
\begin{equation}\label{eq:wps_ehrhart_bound}
L_Q(m)\leq L_P(m),\qquad\text{ for all }m\in\Z_{\geq 0},
\end{equation}
where $L_P$ is determined by the target $\delta$-vector via Lemma~\ref{lem:main_lemma}. The $\delta$-vector for $Q$ can be easily computed~\cite{Kas08c}:
\[
\delta_j:=\abs{\left\{\kappa\in\{0,\ldots,h-1\}\Big|\sum_{i=0}^d\left\{\frac{\lambda_i\kappa}{h}\right\}=j\right\}},\qquad\text{ where }j\in\{0,\ldots,d\}.
\]
Hence the Ehrhart polynomial $L_Q$ can be computed and the condition verified. Notice that the leading coefficient of $L_Q$ is, by construction, at most equal to the leading coefficient of $L_P$ (since this are equal to $\vol{Q}$ and $\vol{P}$ respectively), hence~\eqref{eq:wps_ehrhart_bound} is automatically satisfied for all sufficiently large values of $m$.
\end{enumerate}
In practice, these three conditions are sufficiently strong as to exclude many candidate choices of weights, and this is often sufficient to show that no such $Q$ exists, or to restrict the possibilities to only one or two choices of weights.

Finally, for each choice of weights $(\lambda_0,\ldots,\lambda_d)$ satisfying the above conditions, one can simply work through the possible Hermite normal forms $H$ with $\det{H}=\mult{P}$ and consider the resulting simplicies arising from~\eqref{eq:wps_hermite}. Here one can exploit the symmetries of $Q$ arising from the weights, and the fact that $Q$ is reflexive, in order to reduce the number of choices of $H$ that need to be considered.

\subsection{An example in dimension ten}\label{subsec:dim_10_example}
Let $\delta=(1,1,1,1,9,28,9,1,1,1,1)$. It is easily verified that the corresponding Ehrhart polynomial has a root $\alpha$ with $\Re{\alpha}<-5$ and a root $\beta$ with $\Re{\beta}>4$. We shall use the method sketched above to construct a simplex $P$ with this $\delta$-vector. We will make considerable use of the combinatorics of cyclic quotient singularities; an excellent reference is~\cite{Reid85}.

By~\eqref{wps:poss_mults} we have that $\mult{P}\in\{1,2,3,6,9,18,27,54\}$. Using~\eqref{eq:wps_reflexive} and~\eqref{eq:wps_terminal} and checking conditions~\eqref{item:wps_iff_terminal} and~\eqref{item:wps_divides_degree} we obtain a list of $24$ candidate well-formed Gorenstein terminal weights; see Table~\ref{tab:possible_weights}. Of these, all but $(1,1,1,1,1,1,2,2,2,3,3)=(1^6,2^3,3^2)$ with $\mult{P}=3$ are excluded by condition~\eqref{item:wps_ehrhart_bounds}.

\begin{table}[htdp]
\caption{The initial list of candidate Gorenstein terminal weights from which we may be able to construct a terminal reflexive simplex with $\delta$-vector $(1,1,1,1,9,28,9,1,1,1,1)$.}\label{tab:possible_weights}
\centering
\begin{tabular}{cl}
\toprule
$\mult{P}$&\multicolumn{1}{c}{$(\lambda_0,\ldots,\lambda_{10})$}\\
\cmidrule(lr){1-1}\cmidrule(lr){2-2}
$1$&$(1,1,1,1,1,1,6,6,9,9,18)$\\
$1$&$(1,1,1,1,2,3,3,6,9,9,18)$\\
$1$&$(1,1,1,1,2,3,6,6,6,9,18)$\\
$1$&$(1,1,1,1,2,6,6,9,9,9,9)$\\
$1$&$(1,1,1,2,2,2,3,6,9,9,18)$\\
$1$&$(1,1,1,2,2,2,6,6,6,9,18)$\\
$1$&$(1,1,2,2,2,2,2,6,9,9,18)$\\
$1$&$(1,1,2,2,3,3,3,3,9,9,18)$\\
$1$&$(1,1,2,2,3,3,3,6,6,9,18)$\\
$1$&$(1,1,2,2,3,3,6,6,6,6,18)$\\
$1$&$(1,1,2,2,3,3,6,9,9,9,9)$\\
$1$&$(1,1,2,2,3,6,6,6,9,9,9)$\\
\bottomrule
\end{tabular}
\hspace{1em}
\begin{tabular}{cl}
\toprule
$\mult{P}$&\multicolumn{1}{c}{$(\lambda_0,\ldots,\lambda_{10})$}\\
\cmidrule(lr){1-1}\cmidrule(lr){2-2}
$1$&$(1,1,2,2,6,6,6,6,6,9,9)$\\
$1$&$(1,2,2,2,2,3,3,3,9,9,18)$\\
$1$&$(1,2,2,2,2,3,3,6,6,9,18)$\\
$1$&$(1,2,2,2,2,3,6,9,9,9,9)$\\
$1$&$(1,2,2,2,2,6,6,6,9,9,9)$\\
$1$&$(2,2,2,2,2,2,3,3,9,9,18)$\\
$3$&$(1,1,1,1,1,1,1,1,1,3,6)$\\
$3$&$(1,1,1,1,1,1,1,1,2,2,6)$\\
$3$&$(1,1,1,1,1,1,1,2,3,3,3)$\\
$3$&$(1,1,1,1,1,1,2,2,2,3,3)$\\
$3$&$(1,1,1,1,1,2,2,2,2,2,3)$\\
$3$&$(1,1,1,1,2,2,2,2,2,2,2)$\\
\bottomrule
\end{tabular}
\end{table}

We have reduced to the case where $X=\Proj(1^6,2^3,3^2)/G$, $\abs{G}=3$. Hence $G$ is a cyclic group. Write $\frac{1}{3}(\alpha_0,\ldots,\alpha_{10})$ for a generator of $G$, where $\alpha_i\in\{0,1,2\}$. This acts on $Y=\Proj(1^6,2^3,3^2)$ via
\[
(x_0,\ldots,x_{10})\mapsto(\varepsilon^{\alpha_0}x_0,\ldots,\varepsilon^{\alpha_{10}}x_{10}),\qquad\text{ where }\varepsilon^3=1.
\]
The fixed-points of this group-action are unchanged when we include the action of the weights; that is, the transformation
\[
\frac{1}{3}(\alpha_0,\ldots,\alpha_{10})\mapsto
\frac{1}{3}(\overline{\alpha_0+\lambda_0},\ldots,\overline{\alpha_{10}+\lambda_{10}})
\]
leaves $X$ unchanged. Here $\overline{a}$ denotes the unique integer $0\leq c<3$ such that $a\equiv c\mod 3$. Since $\lambda_0=1$ we use this transformation to arrange for $\alpha_0=0$.

That $\lambda_0=1$ means that the affine chart on $Y$ given by setting $x_0=1$ is smooth. We can interpret this combinatorially as follows. Let $v_0,\ldots,v_{10}\in N'$ be the vertices of $Q$, satisfying $\lambda_0v_0+\ldots+\lambda_{10}v_{10}=0$. Restricting to $x_0=1$ corresponds to taking the cone over the facet $F$ of $Q$ that does not contain $v_0$. That is, $C_0:=\scone{v_1,\ldots,v_{10}}$. That $C_0$ is smooth means that the generators $v_1,\ldots,v_{10}$ form a basis for the lattice $N'$. The action of the group $G$ on this chart is given by $\frac{1}{3}(\alpha_1,\ldots,\alpha_{10})$; we regard this as a rational point $g=\frac{1}{3}\sum_{i=1}^{10}\alpha_iv_i$ in $N'\otimes_\Z\Q$. Let $N'\hookrightarrow N'+g\cdot\Z$ be the natural inclusion of the lattice $N'$ in the lattice $N'+g\cdot\Z$ generated by adding $g$. We identify $C_0$ with its image under this embedding. By assumption $C_0$ is a Gorenstein cone. But $C_0$ is Gorenstein if and only if
\begin{equation}\label{eq:3_divides_sum}
3\,\Big|\,\sum_{i=1}^{10}\lambda_i.
\end{equation}

The Ehrhart polynomials of $P$ and $Q$ are readily calculated from their respective $\delta$-vectors:
\begin{align*}
L_P(m)&=\frac{18}{10!}(3m^{10} + 15m^9 + 225m^8 + 810m^7 + 17969m^6 + 51135m^5 +\\
&\hspace{5em} 274775m^4 + 465240m^3 + 815828m^2 + 591600m + 201600)\\
\text{ and }L_Q(m)&=\frac{18}{10!}(m^{10} + 5m^9 + 255m^8 + 990m^7 + 17843m^6 + 50085m^5 +\\
&\hspace{5em} 274945m^4 + 467560m^3 + 815756m^2 + 590160m + 201600)
\end{align*}
Notice that $L_P(m)=L_Q(m)$ when $m\in\{0,1,2,3\}$. Let $u\in\Hom(N',\Z)$ correspond to the height $1$ supporting hyperplane of the facet $F$. Then
\[
\abs{\{v\in C_0\cap N'\mid u(v)\leq 3\}}=\abs{\{v\in C_0\cap N'+g\cdot\Z\mid u(v)\leq 3\}}.
\]
Hence we conclude that
\begin{equation}\label{eq:age_bound}
\frac{1}{3}\sum_{i=1}^{10}\overline{\kappa\alpha_i}\not\in\{1,2,3\},\qquad\text{ for }\kappa\in\{1,2\}.
\end{equation}

For any permutation $\sigma$ of the integers $\{1,\ldots,10\}$ such that $\lambda_{\sigma i}=\lambda_i$ for each $i\in\{1,\ldots,10\}$, we can regard $\frac{1}{3}(0,\alpha_1,\ldots,\alpha_{10})$ and $\frac{1}{3}(0,\alpha_{\sigma 1},\ldots,\alpha_{\sigma 10})$ as generating equivalent group actions. Thus there are $21$ choices (up to permutation) for $\alpha_1,\ldots,\alpha_5$, and $10$ choices for $\alpha_6,\alpha_7,\alpha_8$. For any particular choice of $\alpha_1,\ldots,\alpha_8$, condition~\eqref{eq:3_divides_sum} means that there are only two choices for $\alpha_9$ and $\alpha_{10}$. Furthermore, whether we pick $\frac{1}{3}(0,\alpha_1,\ldots,\alpha_{10})$ or $\frac{1}{3}(0,\overline{2\alpha_1},\ldots,\overline{2\alpha_{10}})$ as our generator makes no difference: up to permutation this involution fixes exactly $3\cdot2\cdot2=12$ cases (including the trivial action). Remembering to exclude the trivial action, a simple counting argument gives us $(21\cdot10\cdot2-12)/2+11=215$ distinct group actions.

Condition~\eqref{eq:age_bound} reduces the $215$ possible group actions to just $58$ candidates; we call the set of such candidates $\mathcal{G}$. Remembering that our decision to use the weights $(1^6,2^3,3^2)$ to zero $\alpha_0$ was a choice, and that we could just have validly chosen to set $\alpha_i=0$ for any $i\in\{0,\ldots,5\}$ (i.e.~any $i$ such that $\lambda_i=1$), we have that for any $\left<\frac{1}{3}(\alpha_0,\ldots,\alpha_{10})\right>\in\mathcal{G}$ we need $\left<\frac{1}{3}(\overline{\alpha_0-\alpha_i\lambda_0},\ldots,\overline{\alpha_{10}-\alpha_i\lambda_{10}})\right>\in\mathcal{G}$, for each $i\in\{0,\ldots,5\}$, where now we regard groups as being defined only up to permutations $\sigma$ of $\{0,\ldots,10\}$ that fix $\lambda_{\sigma i}=\lambda_i$. This gives us exactly one possible group action, generated by
\[
\frac{1}{3}(0,1,2,0,1,2,0,1,2,1,2).
\]

Finally, we need to compute the resulting polytope $P$ and check that the $\delta$-vector agrees. This is trivial. We find that, up to isomorphism,
\[
P=\sconv{e_1,\ldots,e_9,(1,2,0,1,2,0,1,1,2,3),(-4,-5,-2,-3,-4,-1,-2,-2,-3,-3)},
\]
where $e_i$ is the $i$-th standard basis element, and that $P$ has $\delta$-vector $(1,1,1,1,9,28,9,1,1,1,1)$.

\subsection*{Acknowledgments}
Heged{\"u}s is supported in part by OTKA grant~K77476. Higashitani is partially supported by a JSPS Fellowship for Young Scientists and by JSPS Grant-in-Aid for Young Scientists (B) $\sharp$26800015. Kasprzyk is supported by EPSRC grant~EP/I008128/1 and ERC Starting Investigator Grant number~240123.

\bibliographystyle{plain}

\end{document}